\documentclass[oneside,reqno]{amsart}
\usepackage[T1]{fontenc}
\usepackage{verbatim}
\usepackage{amsthm}
\usepackage{amstext}
\usepackage{amssymb}

\usepackage{wasysym}
\usepackage{esint}
\usepackage{hyperref}

\makeatletter
\numberwithin{equation}{section}
\numberwithin{figure}{section}
\theoremstyle{plain}
\newtheorem{thm}{\protect\theoremname}[section]
\theoremstyle{definition}

\theoremstyle{corollary}

  \theoremstyle{remark}
  \newtheorem{rem}[thm]{\protect\remarkname}
  \theoremstyle{plain}
  \newtheorem{prop}[thm]{\protect\propositionname}
  \theoremstyle{plain}
  \newtheorem{lem}[thm]{\protect\lemmaname}

\usepackage{bbm}

\makeatother

  \providecommand{\lemmaname}{Lemma}
  \providecommand{\propositionname}{Proposition}
  \providecommand{\remarkname}{Remark}
\providecommand{\theoremname}{Theorem}
\providecommand{\definitionname}{Definition}
\providecommand{\corollaryname}{Corollary}

\newcommand\bR{\mathbb{R}}

\begin{document}

\title[Asymptotic Behaviors of Fundamental Solutions]{Asymptotic behaviors of fundamental solution and its derivatives  related to  space-time fractional differential equations}
\begin{abstract}
Let $p(t,x)$ be the fundamental solution to the problem
$$
\partial_{t}^{\alpha}u=-(-\Delta)^{\beta}u, \quad   \alpha\in (0,2), \, \beta\in (0,\infty).
$$
In this paper  we  provide the asymptotic behaviors and sharp upper  bounds of   $p(t,x)$ and its space and time fractional derivatives
$$
D_{x}^{n}(-\Delta_x)^{\gamma}D_{t}^{\sigma}I_{t}^{\delta}p(t,x), \quad  \forall\,\,
n\in\mathbb{Z}_{+}, \,\, \gamma\in[0,\beta],\,\, \sigma, \delta \in[0,\infty),
$$
where $D_{x}^n$ is a partial derivative of order $n$ with respect to $x$, $(-\Delta_x)^{\gamma}$ is  a   fractional Laplace operator and $D_{t}^{\sigma}$ and $I_{t}^{\delta}$
are Riemann-Liouville fractional derivative and integral respectively.

\end{abstract}

\author{Kyeong-Hun Kim and Sungbin Lim}

\address{Department of Mathematics, Korea University, 1 Anam-Dong, Sungbuk-Gu,
Seoul, 136-701, Republic of Korea}
\email{kyeonghun@korea.ac.kr}
\thanks{This work was supported by Samsung Science  and Technology Foundation under Project Number SSTF-BA1401-02}

\address{Department of Mathematics, Korea University, 1 Anam-Dong, Sungbuk-Gu,
Seoul, 136-701, Republic of Korea}
\email{sungbin@korea.ac.kr}


\maketitle

\section{\label{sec:Introduction}Introduction}

Let $\alpha\in(0,2)$,  $\beta\in(0,\infty)$  and $p(t,x)$ be the fundamental solution to the space-time fractional equation
\begin{align}
					\label{eq:space-time FDE}
\partial_{t}^{\alpha}u=\Delta^{\beta}u,\quad(t,x)\in(0,\infty)\times\mathbb{R}^{d}
\end{align}
with $u(0,x)=u_{0}(x)$ (and $\partial_{t}u(0,x)=0$
if $\alpha>1$). Here $\partial_{t}^{\alpha}$ denotes Caputo fractional
derivative and  $\Delta^{\beta}:=-(-\Delta)^{\beta}$ is the fractional Laplacian. The fractional time derivative of order $\alpha\in (0,1)$  can be used
to model the anomalous diffusion exhibiting  subdiffusive behavior,
due to particle sticking and trapping phenomena (see \cite{metzler1999anomalous, metzler2004restaurant}), and the fractional spatial derivative describes long range jumps of particles. The fractional wave equation $\partial_t^{\alpha} u=\Delta u$ with $\alpha\in (1,2)$  governs the propagation of mechanical diffusive waves in viscoelastic media (see \cite{mainardi1995fractional, schneider1989fractional}).  Equation \eqref{eq:space-time FDE}  has been an important topic in the mathematical physics related to non-Markovian diffusion processes with a memory \cite{metzler2000boundary, MK}, in the probability theory related to jump processes \cite{chen2014fractional, chen2012space, meerschaert2014stochastic} and in the theory of differential equations \cite{clement2000schauder, clement2004quasilinear, KKL2014, SY, zacher2005maximal}.

The aim
of ths paper is to present rigorous and self-contained  exposition of   fundamental solution $p(t,x)$. More precisely,
we provide  asymptotic behaviors and upper  bound  of
\begin{align}
				\label{eq:main goal}
D_{x}^{n}(-\Delta_x)^{\gamma}D_{t}^{\sigma}I_{t}^{\delta}p(t,x), \quad  \forall\,\,
n\in\mathbb{Z}_{+}, \,\, \gamma\in[0,\beta],\,\, \sigma, \delta \in[0,\infty),
\end{align}
where $I_{t}^{\delta}$ and $D_{t}^{\sigma}$
denotes the Riemann-Liouville fractional integral and derivative respectively.

 There are certainly considerable works dealing with explicit formula for the fundamental solutions and their asymptotic behaviors (see e.g. \cite{anh2001spectral, EIK, eidelman2004cauchy, gorenflo2000mapping, hanyga2002multi, jo2004precise, kochubei2014asymptotic, li2014asymptotic}). However, only few of them cover the  derivative estimates of the fundamental solutions. In \cite{EIK, eidelman2004cauchy, KKL2014, kochubei2014asymptotic}, upper bounds of \eqref{eq:main goal} were obtained   for  $\beta=1$,  $\sigma=1-\alpha$, and $\gamma=0$. Also, in \cite{jo2004precise}  asymptotic behavior for the  case $\alpha=1,\beta\in(0,1)$ and $\sigma=0$ was obtained.  Note that it is assumed  that either $\alpha=1$ or $\beta=1$ in \cite{EIK, eidelman2004cauchy, jo2004precise,KKL2014, kochubei2014asymptotic}, and moreover spatial fractional derivative $\Delta^{\gamma}p$ and time fractional derivative $D^{\sigma}_tp$ are not obtained in  \cite{EIK, eidelman2004cauchy, KKL2014, kochubei2014asymptotic} and \cite{jo2004precise}  respectively. Our result  substantially improves these results because we only assume
$\alpha\in (0,2)$ and $\beta\in (0,\infty)$ and  we provide two sides estimates of   both  space and time fractional derivatives of arbitrary order.

Our approach relies also on the properties of some  special functions including the Fox H functions.
However, unlike most of previous works, we do not use the method of Mellin, Laplace and inverse Fourier transforms which, due to the non-exponential decay at infinity of the Fox H functions, require  restrictions  on  the space dimension $d$ and other parameters $\beta, \delta, \gamma$, and $\sigma$ (see Remark \ref{rem: integrable}).

Below we give two main applications of our results.  First, consider the  non-homogeneous
fractional evolution equation
\begin{align}
				\label{eq:non-homogeneous FDE}
\partial_{t}^{\alpha}u=\Delta^{\beta}u+f,\quad(t,x)\in(0,\infty)\times\mathbb{R}^{d}
\end{align}
with $u(0,x)=0$ and additionally $\partial_{t}u(0,x)=0$
if $\alpha>1$.  One can show  (see e.g. \cite{EIK, KKL2014, kochubei2014asymptotic}) that the solution to the problem is given by
$$
\mathcal{G}f(t,x)=\int_{0}^{t}\int_{\mathbb{R}^{d}}Q_{\alpha}(t-s,x-y)f(s,y)dyds
$$
where
$$
Q_{\alpha}(t,x):=\begin{cases}
D_{t}^{1-\alpha}p(t,x) & :\alpha\in(0,1)\\
I_{t}^{\alpha-1}p(t,x) & :\alpha\in(1,2).
\end{cases}
$$
It turns out (see \cite{kim2015l_q,Kim2014BMOpseudo})) that to obtain the  $L_q(L_p)$-estimate
\begin{align}
					\label{eq:Lp estimate}
\left\Vert \Delta^{\beta}\mathcal{G}f\right\Vert _{L_{q}((0,T),L_{p}(\mathbb{R}^{d}))}\leq N\|f\|_{L_{q}((0,T),L_{p}(\mathbb{R}^{d}))}, \quad p,q>1
\end{align}
it is sufficient to show that for  $a,b>0$
$$
\int_{0}^{a}\int_{|x|\geq b}|\Delta^{\beta}Q_{\alpha}(t,x)|dxdt\leq N\frac{a^{\alpha}}{b^{2\beta}},
$$
$$
\int_{a}^{\infty}\int_{\mathbb{R}^{d}}\left|\partial_{t}\Delta^{\beta}Q_{\alpha}(t,x)\right|dxdy\leq\frac{N}{a},\quad\int_{a}^{\infty}\int_{\mathbb{R}^{d}}\left|\nabla_{x}\Delta^{\beta}Q_{\alpha}(t,x)\right|dxdt\leq Na^{-\frac{\alpha}{2\beta}}.
$$
One can use our estimates to prove the above three inequalities, and therefore  \eqref{eq:Lp estimate}  can  be obtained as a corollary.
Our second application is    the $L_{p}$-theory of the stochastic partial differential equations of the type
$$
\partial_{t}^{\alpha}u = \Delta u + \partial_{t}^{\alpha+\sigma}\int_{0}^{t}g(s,x) dW_{s},
$$
where  $ \sigma < \frac{1}{2}$ and  $W_t$ is a Wiener process defined on a probability space $(\Omega, dP)$.  One can show (see \cite{chen2014fractional})  that the solution to this problem is given by the formula
$$
u(t,x)=\int_{0}^{t} \left(\int_{\mathbb{R}^{d}}P_{\sigma}(t-s,x-y)g(s,y)dy\right) dW_s.
$$
Here $P_{\sigma}(t,x)$ is defined as
$$
P_{\sigma}(t,x):=\begin{cases}
I_{t}^{|\sigma|}p(t,x) & :\sigma\leq0\\
D_{t}^{|\sigma|}p(t,x) & :\sigma>0.
\end{cases}
$$
As has been shown for the case $\alpha=1$ (see \cite{Kim2012, kim2013parabolic, Krylov2010, van2012stochastic}),  sharp estimates of $D^n_x \Delta^{\gamma}P_{\sigma}$ can be used to obtain
$L_p$-estimate
\begin{align}
				\label{stochastic maximal}
\left\Vert \Delta^{\gamma}u\right\Vert _{L_{p}(\Omega\times (0,T)\times\mathbb{R}^{d})}\leq N\|g\|_{L_{p}(\Omega\times (0,T)\times\mathbb{R}^{d})}.
\end{align}
The detail of \eqref{stochastic maximal}  will be given for $\gamma \leq (2 \wedge \frac{1-2\sigma}{\alpha})$ in a subsequent paper.

The rest of the article is organized as follows. In Section \ref{sec:Main-results}
we state our main results, Theorems \ref{thm:main result 1}, \ref{thm:derivative of kernel}, and \ref{cor:main result 1}.
In Section \ref{sec:Preliminaries} we present the definition of the
Fox H functions and their several properties. For the convenience
of the reader, we repeat the relevant material and demonstration in
\cite{kilbas2004h} and \cite{KKL2014}, thus making our exposition
self-contained. Section \ref{sec:Auxiliary-results} contains asymptotic
behaviors at zero and infinity of the Fox H function.  In Section \ref{sec:representation of solution}
we present explicit representation of fundamental solutions and their
fractional and classical derivatives. Finally, in Section \ref{sec: proof of main theorem}
we prove our main results.

We finish the introduction with some notion used in this article.
We write $f\apprle g$ for $|x|\leq \delta$ (resp. $|x|\geq \delta$)  if there exists a positive
constant $C$ independent of $x$ such that $f(x)\leq Cg(x)$ for $|x|\leq \delta$ (resp. $|x|\geq \delta$),  and $f\sim g$ for $|x|\leq \delta$ (resp. $|x|\geq \delta$) if
 $f\apprle g\apprle f$ for $|x|\leq \delta$ (resp. $|x|\geq \delta$).  We say $f\sim g$  as $|x|\to 0$ (resp. $|x|\to \infty$) if there exists $\varepsilon\in (0,1)$ such that $f\sim g$ for $|x|\leq \varepsilon$ (resp. $|x|\geq \varepsilon^{-1}$).
  We write $f(x)=O(g(|x|))$ as
$|x|\rightarrow0$ (resp. $|x|\to \infty$) if there exists $\delta>0$ such that
$|f(x)|\apprle|g(|x|)|$ for $|x|<\delta$ (resp. $|x|\geq \delta$).
 We use ``:='' to denote a definition. As usual $\mathbb{R}^{d}$
stands for the Euclidean space of points $x=(x^{1},\ldots,x^{d})$,
 $\mathbb{R}_{0}^{d}:=\mathbb{R}^{d}\setminus\{0\}$, and  $\mathbb{Z}_{+}:=\{0,1,2,\cdots\}$. For multi-indices $\mathfrak{a}=(a_{1},\ldots,a_{d})\in\mathbb{Z}^d_{+}$, $n\in \mathbb{N}$ and functions $u(x)$ we set
$$
D_{i}u=\frac{\partial u}{\partial x^{i}},\quad D_{x}^{\mathfrak{a}}u=D_{1}^{a_{1}}\cdots D_{d}^{a_{d}}u, \quad
D^n_x:=\{D^{\mathfrak{a}}_x :|\mathfrak{a}|=n\}.
$$
 $\lfloor a\rfloor$
is the biggest integer which is less than or equal to $a$. By $\mathcal{F}$
 we denote the $d$-dimensional Fourier transform, that is,
$$
\mathcal{F}\{f\}(\xi):=\int_{\mathbb{R}^{d}}e^{-\textnormal{i}(x,\xi)}f(x)dx.
$$
For a complex number $z$, $\Re[z]$ and $\Im[z]$ are the real part
and imaginary part of $z$ respectively.

\section{\label{sec:Main-results}Main results}

We first introduce some definitions related to the fractional calculus. Let $\beta\geq 0$.
For a function $u\in L_{1}(\mathbb{R}^{d})$, we write $\Delta^{\beta}u=f$
if  there exists a function $f\in L_{1}(\mathbb{R}^{d})$ such that
$$
\mathcal{F}\left\{ f(\cdot)\right\} (\xi)=|\xi|^{2\beta}\mathcal{F}\left\{ u\right\} (\xi).
$$
For $u\in L_{1}((0,T))$, the Riemann-Liouville fractional integral
of the order $\alpha\in(0,\infty)$ is defined as
$$
I_{t}^{\alpha}u:=\int_{0}^{t}(t-s)^{\alpha-1}u(s)ds, \quad t\leq T.
$$
One can easily check
\begin{equation}
                                                          \label{eqn 4.15.3}
I^{\alpha}_tI^{\beta}_t=I^{\alpha+\beta}_t, \quad \forall \, \alpha,\beta\geq 0.
\end{equation}

Let $n\in \mathbb{N}$ and $n-1\leq \alpha <n$. The Riemann-Liouville fractional derivative $D_{t}^{\alpha}$ and the Caputo fractional derivative $\partial_{t}^{\alpha}$ are defined as
\begin{equation}
                          \label{eqn 4.15}
D_{t}^{\alpha}u:=\left(\frac{d}{dt}\right)^{n}\left(I_{t}^{n-\alpha}u\right)
\end{equation}
$$
\partial^{\alpha}_tu:=D^{\alpha-(n-1)}_t\left(u^{(n-1)}(t)-u^{(n-1)}(0)\right).
$$
By definition \eqref{eqn 4.15} for any $\alpha\geq 0$ and $u\in L_1((0,T))$,
\begin{align}
				\label{eqn 4.20.1}
D^{\alpha}_t I_{t}^{\alpha} u=u.
\end{align}
Using \eqref{eqn 4.15.3}-\eqref{eqn 4.20.1}, one can check
\begin{align*}
\partial_{t}^{\alpha}u= D_{t}^{\alpha} \left(u(t)-\sum_{k=0}^{n-1}\frac{t^{k}}{k!}u^{(k)}(0)\right), \quad \quad n-1\leq \alpha <n.
\end{align*}
Thus $D^{\alpha}_tu=\partial^{\alpha}_t u$ if $u(0)=u^{(1)}(0)=\cdots = u^{(n-1)}(0)=0$. For more information on the fractional derivatives, we refer the reader to \cite{Po, SKM}.

For $\sigma\in\mathbb{R}$ we define Riemann-Liouville fractional operator $\mathbb{D}^{\alpha}_t$ as
$$
\mathbb{D}_{t}^{\sigma}:=\begin{cases}
D_{t}^{|\sigma|} & :\sigma>0\\
I_{t}^{|\sigma|} & :\sigma<0.
\end{cases}
$$
Then by \eqref{eqn 4.15.3} and \eqref{eqn 4.15}, for any
 $\alpha,\beta\geq0$
$$
D_{t}^{\alpha}I_{t}^{\beta}=\mathbb{D}_{t}^{\alpha-\beta}.
$$

The Mittag-Leffler function $E_{\alpha,\beta}(z)$ is defined as
\begin{equation}
         \label{mittag}
E_{\alpha,\beta}(z):=\sum_{k=0}^{\infty}\frac{z^{k}}{\Gamma(\alpha k+\beta)},\quad\alpha,\beta\in\mathbb{C},\Re[\alpha]>0
\end{equation}
for $z\in\mathbb{C}$ and we write $E_{\alpha}(z)=E_{\alpha,1}(z)$
for short. Using the equality
$$
\mathbb{D}^{\sigma}_t t^{\alpha}= \frac{\Gamma(\alpha+1)}{\Gamma(\alpha+1-\sigma)}t^{\alpha-\sigma}, \quad \forall \, \sigma\in \mathbb{R}, \alpha\geq 0
$$
one can check   that for $t>0$
$$
\mathbb{D}_{t}^{\sigma}E_{\alpha}(-\lambda t^{\alpha})=t^{-\sigma}E_{\alpha,1-\sigma}(-\lambda t^{\alpha}),\quad\sigma\in\mathbb{R},\lambda\geq0,
$$
and that  for any constant $\lambda$,
$$
\varphi(t):=E_{\alpha}(-\lambda t^{\alpha})
$$
satisfies $\varphi(0)=1$ (also $\varphi'(0)=0$ if $\alpha>1$) and
$$
\partial_{t}^{\alpha}\varphi=-\lambda\varphi,\quad t>0.
$$
Let $\alpha\in(0,2)$ and $\beta\in(0,\infty)$. By taking the Fourier
transform to the equation
$$
\partial_{t}^{\alpha}u=\Delta^{\beta}u,\quad t>0,\quad u(0,x)=u_{0}(x),\quad(\mbox{and }u'(0,x)=0\ \mbox{if}\ \mbox{\ensuremath{\alpha}>1})
$$
one can formally get
$$
\mathcal{F}\left\{ u(t,\cdot)\right\} =E_{\alpha}(-|\xi|^{2\beta}t^{\alpha})\mathcal{F}\left\{u_{0}\right\}.
$$
Therefore, to obtain the fundamental solution, it is needed to
find an integrable function $p(t,x)\in L_{1}(\mathbb{R}^{d})$ satisfying
\begin{align}
					\label{eq:main goal 2}
\mathcal{F}\{p(t,\cdot)\}=E_{\alpha}(-|\xi|^{2\beta}t^{\alpha}).
\end{align}

Denote
$$
\mathbf{M}(t,x):=|x|^{2\beta}t^{-\alpha}.
$$
 In the following  theorems we give the asymptotic behaviors of $D^n_x \Delta^{\gamma}\mathbb{D}^{\sigma}_t p(t,x)$ as $\mathbf{M} \to 0$ and $\mathbf{M}\to \infty$. We also provide upper bounds when $\mathbf{M}\leq 1$ and $\mathbf{M}\geq 1$.

 Firstly, we consider the case  $\mathbf{M}\to\infty$.

\begin{thm}
									\label{thm:main result 1}
Let $\alpha\in(0,2)$, $\beta\in(0,\infty)$, $\gamma\in[0,\infty)$, $\sigma\in \bR$ and $n\in\mathbb{N}$. There exists a function $p(t,\cdot)\in L_{1}(\mathbb{R}^{d})$ satisfying \eqref{eq:main goal 2}.  Furthermore,  the following asymptotic behaviors hold as $\mathbf{M} \to \infty$:

(i) If $\beta\in{\mathbb{N}}$, then  for some constant $c>0$ depending only on $d,n,\alpha,\beta,\sigma$
\begin{align*}
\left|\mathbb{D}_{t}^{\sigma}p(t,x)\right|\apprle |x|^{-d}t^{-\sigma}\exp\left\{-c|x|^{\frac{2\beta}{2\beta-\alpha}}t^{-\frac{\alpha}{2\beta-\alpha}}\right\}
\end{align*}
and
\begin{align}
					\label{eq:asy exp large-1}
\left|D_{x}^{n}\mathbb{D}_{t}^{\sigma}p(t,x)\right|  \apprle  |x|^{-d-n}t^{-\sigma}\exp\left\{-c|x|^{\frac{2\beta}{2\beta-\alpha}}t^{-\frac{\alpha}{2\beta-\alpha}}\right\}.
\end{align}

(ii) If $\alpha=1$, $\beta\notin\mathbb{N}$, and $\sigma=0$,
\begin{align}
				\label{eq:asy 0 large}
\left|\Delta^{\gamma}p(t,x)\right|\sim\begin{cases}
t|x|^{-d-2\gamma-2\beta}& :\gamma\in \mathbb{Z}_{+} \\
|x|^{-d-2\gamma} & :\gamma\in [0,\infty)\setminus \mathbb{Z}_{+}
\end{cases}
\end{align}
and
\begin{align*}
\left|D_{x}^{n}\Delta^{\gamma}p(t,x)\right|\apprle\begin{cases}
t|x|^{-d-2\gamma-2\beta-n}& :\gamma\in \mathbb{Z}_{+} \\
|x|^{-d-2\gamma-n} & :\gamma\in [0,\infty)\setminus \mathbb{Z}_{+}.
\end{cases}
\end{align*}

(iii)  If $\gamma\in(0,\beta)\setminus\mathbb{N}$,
\begin{align*}
|\Delta^{\gamma}\mathbb{D}_{t}^{\sigma}p(t,x)|\sim |x|^{-d-2\gamma}t^{-\sigma}
\end{align*}
and
\begin{align*}
|D_{x}^{n}\Delta^{\gamma}\mathbb{D}_{t}^{\sigma}p(t,x)|\apprle |x|^{-d-2\gamma-n}t^{-\sigma}.
\end{align*}

(iv)  If $\beta\notin\mathbb{N}$ and $\gamma\in [0,\beta)\cap\mathbb{Z}_{+}$,
\begin{align}
				\label{eq:asy 2 large}
\left|\Delta^{\gamma}\mathbb{D}_{t}^{\sigma}p(t,x)\right|\sim|x|^{-d-2\gamma-2\beta}t^{-\sigma+\alpha}
\end{align}
and
\begin{align*}
\left|D_{x}^{n}\Delta^{\gamma}\mathbb{D}_{t}^{\sigma}p(t,x)\right|\apprle|x|^{-d-2\gamma-2\beta-n}t^{-\sigma+\alpha}.
\end{align*}

 (v) If $\gamma=\beta$ and $d\geq2$,
\begin{align}
				\label{eq:asy 3 large}
\left|\Delta^{\beta}\mathbb{D}_{t}^{\sigma}p(t,x)\right|\sim\begin{cases}
|x|^{-d-4\beta}t^{-\sigma+\alpha} & :\beta\in\mathbb{N}\textnormal{ or }\sigma\in\mathbb{N}\\
|x|^{-d-2\beta}t^{-\sigma} & : \textnormal{otherwise}
\end{cases}
\end{align}
and
\begin{align}
				\label{eq:asy 3-1 large}
\left|D_{x}^{n}\Delta^{\beta}\mathbb{D}_{t}^{\sigma}p(t,x)\right|\apprle\begin{cases}
|x|^{-d-4\beta-n}t^{-\sigma+\alpha} & :\beta\in\mathbb{N}\textnormal{ or }\sigma\in\mathbb{N}\\
|x|^{-d-2\beta-n}t^{-\sigma} & : \textnormal{otherwise}.
\end{cases}
\end{align}

(vi)  If $\gamma=\beta$, $\sigma+\alpha\in\mathbb{N}$, and $d=1$, then
\eqref{eq:asy 3 large} and \eqref{eq:asy 3-1 large} hold.

\end{thm}

\begin{rem}
(i) Note that $D_{x}^{n}\mathbb{D}_{t}^{\sigma}p(t,x)$ has exponential decay as $\mathbf{M}\to\infty$ only when $\beta$ is a positive integer.

(ii) Let $x\neq 0$. Then by \eqref{eq:asy exp large-1} and \eqref{eq:asy 2 large} with $\gamma=0$,  $D^n_x\mathbb{D}_{t}^{\sigma}p(t,x)\rightarrow 0 $ as $t\rightarrow 0$ if either $\beta$ is a positive integer or  $\sigma<\alpha$.

(iii) If $\gamma=\beta$ and $d=1$, then we additionally  assumed  $\sigma+\alpha\in \mathbb{N}$.  Without this extra condition we had a trouble in using Fubini's theorem in our proof.

(iv)  Note that we have only upper bounds of $D^n_x \Delta^{\gamma}\mathbb{D}^{\sigma}_tp(t,x)$ unless $n=0$. This is because, for instance, $D_ip(t,x)$ is of type $x^i g(t,x)$ (see \eqref{eq:fundamental solution}) and  becomes zero if $x^i=0$. Hence  we can not have positive lower bound of $D_ip(t,x)$ for such $x$.
\end{rem}

Secondly, we consider the case $\mathbf{M}\rightarrow 0$.

\begin{thm}
								\label{thm:derivative of kernel}
Let $\alpha,\beta,\gamma,\sigma$ be given as in Theorem \ref{thm:main result 1} and $n\in \mathbb{N}$. Then the following asymptotic behaviors hold as $\mathbf{M}\rightarrow 0$:

(i) If  $\gamma\in[0,\beta)$ and $\sigma+\alpha\notin \mathbb{N}$,
\begin{align}	
				\label{eq:asy 1 small}
\left|\Delta^{\gamma}\mathbb{D}_{t}^{\sigma}p(t,x)\right|\sim\begin{cases}
t^{-\sigma-\frac{\alpha(d+2\gamma)}{2\beta}} & :\gamma<\beta-\frac{d}{2}\\
|x|^{-d-2\gamma+2\beta}t^{-\sigma-\alpha}\left(1+|\ln |x|^{2\beta}t^{-\alpha}|\right) & :\gamma=\beta-\frac{d}{2}\\
|x|^{-d-2\gamma+2\beta}t^{-\sigma-\alpha} & :\gamma>\beta-\frac{d}{2}
\end{cases}
\end{align}
and
\begin{align}
				                               \label{eq:asy 2 small}
|D_{x}^{n}\Delta^{\gamma}\mathbb{D}_{t}^{\sigma}p(t,x)|\apprle\begin{cases}
|x|^{2-n}t^{-\sigma-\frac{\alpha(d+2\gamma+2)}{2\beta}} & :\gamma<\beta-\frac{d}{2}-1\\
|x|^{2-n}t^{-\sigma-\alpha}(1+|\ln |x|^{2\beta}t^{-\alpha}|) & :\gamma=\beta-\frac{d}{2}-1\\
|x|^{-d-2\gamma+2\beta-n}t^{-\sigma-\alpha} & :\gamma>\beta-\frac{d}{2}-1.
\end{cases}
\end{align}

(ii)  If $\gamma\in [0,\beta)$ and $\sigma+\alpha\in\mathbb{N}$,
\begin{align}
					\label{eq:asy 3 small}
\left|\Delta^{\gamma}\mathbb{D}_{t}^{\sigma}p(t,x)\right|\sim\begin{cases}
t^{-\sigma-\frac{\alpha(d+2\gamma)}{2\beta}} & :\gamma<2\beta-\frac{d}{2}\\
|x|^{-d-2\gamma+4\beta}t^{-\sigma-2\alpha}\left(1+\left|\ln |x|^{2\beta}t^{-\alpha}\right|\right) & :\gamma=2\beta-\frac{d}{2}\\
|x|^{-d-2\gamma+4\beta}t^{-\sigma-2\alpha} & :\gamma>2\beta-\frac{d}{2}
\end{cases}
\end{align}
and
\begin{align}
				\label{eq:asy 4 small}
\left|D_{x}^{n}\Delta^{\gamma}\mathbb{D}_{t}^{\sigma}p(t,x)\right|\apprle\begin{cases}
|x|^{2-n}t^{-\sigma-\frac{\alpha(d+2\gamma+2)}{2\beta}}&:\gamma<2\beta-\frac{d}{2}-1\\
|x|^{2-n}t^{-\sigma-2\alpha}(1+\left|\ln|x|^{2\beta}t^{-\alpha}\right|) & :\gamma=2\beta-\frac{d}{2}-1 \\
|x|^{-d-2\gamma+4\beta-n}t^{-\sigma-2\alpha} & :\gamma>2\beta-\frac{d}{2}-1.
\end{cases}
\end{align}

(iii) If $\alpha=1$ and $\sigma=0$,

\begin{align*}
\left|\Delta^{\gamma}p(t,x)\right| \sim t^{-\frac{d+2\gamma}{2\beta}},\quad \left|D_{x}^{n}\Delta^{\gamma}p(t,x)\right| \apprle |x|^{2-n}t^{-\frac{d+2\gamma+2}{2\beta}}.
\end{align*}

(iv) If $\gamma=\beta$ and $d\geq2$,
\begin{align}
				\label{eq:asy 6 small}
\left|\Delta^{\beta}\mathbb{D}_{t}^{\sigma}p(t,x)\right|\sim\begin{cases}
t^{-\sigma-\alpha-\frac{\alpha d}{2\beta}} & :\frac{d}{2}<\beta\\
|x|^{-d+2\beta}t^{-\sigma-2\alpha}\left(1+\left|\ln|x|^{2\beta}t^{-\alpha}\right|\right) & :\frac{d}{2}=\beta\\
|x|^{-d+2\beta}t^{-\sigma-2\alpha} & :\frac{d}{2}>\beta
\end{cases}
\end{align}
and
\begin{align}
				\label{eq:asy 7 small}
\left|D_{x}^{n}\Delta^{\beta}\mathbb{D}_{t}^{\sigma}p(t,x)\right|\apprle\begin{cases}
|x|^{2-n}t^{-\sigma-\alpha-\frac{\alpha(d+2)}{2\beta}} & :\frac{d}{2}+1<\beta \\
|x|^{2-n}t^{-\sigma-2\alpha}(1+\left|\ln|x|^{2\beta}t^{-\alpha}\right|) & :\frac{d}{2}+1=\beta\\
|x|^{-d+2\beta-n}t^{-\sigma-2\alpha} & :\frac{d}{2}+1>\beta.
\end{cases}
\end{align}

 (v) If $\gamma=\beta$, $\sigma+\alpha\in\mathbb{N}$, and $d=1$, then \eqref{eq:asy 6 small} and \eqref{eq:asy 7 small} hold.

\end{thm}

Next we give  the     upper estimates when $\mathbf{M}\geq 1$ and $\mathbf{M}\leq 1$.

\begin{thm}
									\label{cor:main result 1}
(i) Assertions (i)-(iv) of Theorem \ref{thm:main result 1} also hold for $\mathbf{M}\geq 1$ if  ``$\sim$'' is replaced by ``$\apprle$".

(ii) Assertions (i)-(v) of Theorem \ref{thm:derivative of kernel} also hold for $\mathbf{M}\leq 1$ if  ``$\sim$'' is replaced by ``$\apprle$".

\end{thm}

The proofs of Theorems  \ref{thm:main result 1}, \ref{thm:derivative of kernel} and \ref{cor:main result 1}  are  given in Section \ref{sec: proof of main theorem}.

\begin{rem}
Let $\alpha=1$ and $\beta\in (0,\infty)$, and $\sigma=0$. Then by Theorem \ref{cor:main result 1},
$$
\left|D_{x}^{n}\Delta^{\gamma}p(t,x)\right|\apprle\begin{cases}
t|x|^{-d-2\beta-2\gamma-n} & :\gamma\in\mathbb{Z}_{+}\\
|x|^{-d-2\gamma-n} & :\gamma\notin\mathbb{Z}_{+}
\end{cases}
$$
holds for $|x|^{2}\geq t$. Also, by Theorem  \ref{cor:main result 1}, for $|x|^{2}\leq t$,
$$
\left|\Delta^{\gamma}p(t,x)\right| \apprle t^{-\frac{d+2\gamma}{2\beta}},\quad \left|D_{x}^{n}\Delta^{\gamma}p(t,x)\right| \apprle |x|^{2-n}t^{-\frac{d+2\gamma+2}{2\beta}}.
$$
These estimates cover the results of \cite[Lemma 3.1, 3.3]{Kim2012} and \cite[Corollary 1]{jo2004precise}.
\end{rem}

\begin{rem}
Theorem \ref{cor:main result 1}
also implies the results of \cite[Proposition 5.1, 5.2]{EIK}.
Let $\beta=1$ and take a $|n|\geq1$. For $|x|^{2} \geq t^{\alpha}$.
\begin{align*}
\left|D_{x}^{n}\mathbb{D}_{t}^{\sigma}p(t,x)\right| &\apprle |x|^{-d-n}t^{-\sigma}\exp\left\{-c t^{-\frac{\alpha}{2-\alpha}}|x|^{\frac{2}{2-\alpha}}\right\}.
\\ &\apprle t^{-\frac{\alpha (d+n)}{2}-\sigma}\exp\left\{-c t^{-\frac{\alpha}{2-\alpha}}|x|^{\frac{2}{2-\alpha}}\right\}.
\end{align*}
Also, (cf. \eqref{eq:asy 1 small}, \eqref{eq:asy 2 small}, and \eqref{eq:asy 4 small}),
$$
|p(t,x)|\apprle\begin{cases}
t^{-\frac{\alpha d}{2}} & :d\geq 3\\
t^{-\alpha}(1+|\ln|x|^{2}t^{-\alpha}|) & :d=2\\
|x|^{-d+2}t^{-\alpha} & :d=1
\end{cases}
$$
and
$$
\left|D_{x}^{n}p(t,x)\right|\apprle|x|^{-d+2-n}t^{-\alpha},\quad \left|D_{x}^{n}\Delta p(t,x)\right|\apprle|x|^{-d+2-n}t^{-2\alpha},
$$
$$
\left|D_{x}^{n}\mathbb{D}_{t}^{1-\alpha}p(t,x)\right|\apprle\begin{cases}
|x|^{-d+4-n}t^{-\alpha-1} & :d\geq 3\\
|x|^{2-n}t^{-1}(1+|\ln|x|^{2}t^{-\alpha}|) & :d=2\\
|x|^{1-n}t^{-1} & :d=1
\end{cases}
$$
hold for $|x|^{2}\leq t^{\alpha}$.
\end{rem}

\section{\label{sec:Preliminaries}The Fox H function}

In this section,  we introduce the definition and some properties of  the Fox H function. We refer to \cite{kilbas2004h} for further information.

\subsection{Definition}

Let $\Gamma(z)$ denote the gamma function which can be defined (see
\cite[Section 1.1]{Gamma}) for $z\in\mathbb{C}\setminus\{0,-1,-2,\ldots\}$
as
$$
\Gamma(z)=\lim_{n\rightarrow\infty}\frac{n!n^{z}}{z(z+1)\cdots(z+n)}.
$$
Note that $\Gamma(z)$ is a meromorphic function with simple poles
at the nonpositive integers. From the definition, for $z\in\mathbb{C}\setminus\{0,-1,-2,\ldots\}$,
\begin{align}
					\label{eq:zGamma(z)=Gamma(z+1)}
z\Gamma(z)=\Gamma(z+1),
\end{align}
and it holds that
\begin{align}
					\label{eq:gamma property}
\Gamma(1-z)\Gamma(z)=\frac{\pi}{\sin\pi z},\quad \prod_{k=0}^{m-1}\Gamma(z+\frac{k}{m})=(2\pi)^{\frac{m-1}{2}}m^{\frac{1}{2}-mz}\Gamma(mz).
\end{align}
One can easily check that for $k\in\mathbb{Z}_{+}$,
$$
\textnormal{Res}_{z=-k}[\Gamma(z)]=\lim_{z\rightarrow-k}(z+k)\Gamma(z)=\frac{(-1)^{k}}{k!},
$$
where $\textnormal{Res}_{z=z_{0}}[f(z)]$ denotes the residue of $f(z)$
at $z=z_{0}$. From the Stirling's approximation
$$
\Gamma(z)\sim\sqrt{2\pi}e^{(z-\frac{1}{2})\log z}e^{-z},\quad|z|\rightarrow\infty,
$$
it follows that
\begin{align}
					\label{eq:gamma-horizantal}
|\Gamma(a+\textnormal{i}b)|\sim\sqrt{2\pi}|a|^{a-\frac{1}{2}}e^{-a-\pi[1-\textnormal{sign}(a)]b/2}, \quad  |a|\rightarrow\infty
\end{align}
and
\begin{align}
					\label{eq:gamma-vertical}
|\Gamma(a+\textnormal{i}b)|\sim\sqrt{2\pi}|b|^{a-\frac{1}{2}}e^{-a-\frac{\pi|b|}{2}}, \quad  |b|\rightarrow\infty.
\end{align}

Let $m,n,\nu,\mu$ be fixed integers satisfying $0\leq m\leq\mu$,
$0\leq n\leq\nu$. Assume that real parameters $\mathfrak{c}_{1},\ldots,\mathfrak{c}_{\nu}$,
$\mathfrak{d}_{1},\ldots,\mathfrak{d}_{\mu}$ and positive real parameters
$\gamma_{1},\ldots,\gamma_{\nu}$, $\delta_{1},\ldots,\delta_{\mu}$
are given such that
\begin{align}
				\label{eq:P1 P2}
\max_{1\leq j\leq m}\left(-\frac{\mathfrak{d}_{j}}{\delta_{j}}\right)<\min_{1\leq j\leq n}\left(\frac{1-\mathfrak{c}_{j}}{\gamma_{j}}\right).
\end{align}
For each $k\in\mathbb{Z}_{+}$, we set
$$
\mathfrak{c}_{j,k}=\frac{1-\mathfrak{c}_{j}+k}{\gamma_{j}},\quad\mathfrak{d}_{j,k}=-\frac{\mathfrak{d}_{j}+k}{\delta_{j}},
$$
which constitute
$$
P_{1}=\left\{ \mathfrak{d}_{j,k}\in\mathbb{R}:j\in\{1,\ldots,m\},k\in\mathbb{Z}_{+}\right\} ,
$$
$$
P_{2}=\left\{ \mathfrak{c}_{j,k}\in\mathbb{R}:j\in\{1,\ldots,n\},k\in\mathbb{Z}_{+}\right\} .
$$
Note that $P_{1}\cap P_{2}=\emptyset$ by (\ref{eq:P1 P2}).  We
arrange the elements of $P_{1}$ and $P_{2}$  as follows:
\begin{align}
					\label{eq:rearrangement}
P_{1}=\left\{\hat{\mathfrak{d}}_{0}  >\hat{\mathfrak{d}}_{1}>\hat{\mathfrak{d}}_{2}>\cdots\right\} ,\quad P_{2}=\left\{\hat{\mathfrak{c}}_{0}<\hat{\mathfrak{c}}_{1}<\hat{\mathfrak{c}}_{2}<\cdots\right\} .
\end{align}
For the above parameters, define
$$
\mathcal{H}(z):=\frac{\prod_{j=1}^{m}\Gamma(\mathfrak{d}_{j}+\delta_{j}z)\prod_{j=1}^{n}\Gamma(1-\mathfrak{c}_{j}-\gamma_{j}z)}{\prod_{j=n+1}^{\nu}\Gamma(\mathfrak{c}_{j}+\gamma_{j}z)\prod_{j=m+1}^{\mu}\Gamma(1-\mathfrak{d}_{j}-\delta_{j}z)}.
$$
Note that $P_{1}$ and $P_{2}$ are sets of poles of $\mathcal{H}(z)$. To
describe the behavior of $\mathcal{H}(z)$ as $|z|\rightarrow\infty$,
we set
$$
\alpha^{*}:=\sum_{i=1}^{n}\gamma_{i}-\sum_{i=n+1}^{\nu}\gamma_{i}+\sum_{j=1}^{m}\delta_{j}-\sum_{j=m+1}^{\mu}\delta_{j},
$$
and
$$
\Lambda:=\sum_{j=1}^{\mu}\mathfrak{d}_{j}-\sum_{j=1}^{\nu}\mathfrak{c}_{j}+\frac{\nu-\mu}{2},\quad\omega:=\sum_{j=1}^{\mu}\delta_{j}-\sum_{j=1}^{\nu}\gamma_{i},\quad\eta:=\prod_{j=1}^{\nu}\gamma_{j}^{-\gamma_{j}}\prod_{j=1}^{\mu}\delta_{j}^{\delta_{j}}.
$$
Due to (\ref{eq:gamma-horizantal}),
\begin{align}
				\label{eq:Stirling approx horizantal+}
|\mathcal{H}(a+\textnormal{i}b)r^{-a-\textnormal{i}b}|\sim\left(\frac{e}{a}\right)^{-\omega a}\left(\frac{\eta}{r}\right)^{a}a^{\Lambda},\quad r\in(0,\infty)
\end{align}
as $a\rightarrow\infty$ and
\begin{align}
					\label{eq:Stirling approx horizantal-}
|\mathcal{H}(a+\textnormal{i}b)r^{-a-\textnormal{i}b}|\sim\left(\frac{e}{|a|}\right)^{\omega|a|}\left(\frac{r}{\eta}\right)^{-|a|}|a|^{\Lambda},\quad r\in(0,\infty)
\end{align}
as $a\rightarrow-\infty$. By (\ref{eq:gamma-vertical}), it follows
that
\begin{align}
					\label{eq:Stirling approx vertical}
|\mathcal{H}(a+\textnormal{i}b)r^{-a-\textnormal{i}b}|\sim\left(\frac{e}{|b|}\right)^{-\omega a}\left(\frac{\eta}{r}\right)^{a}|b|^{\Lambda}e^{-\alpha^{*}|b|\pi/2},\quad r\in(0,\infty)
\end{align}
as $|b|\rightarrow\infty$.

The Fox H funciton $\textnormal{H}_{\nu\mu}^{mn}(r)$ $(r>0)$ is defined
via Mellin-Barnes type integral in the form
\begin{align}
\textnormal{H}_{\nu\mu}^{mn}(r) & :=\textnormal{H}_{\nu\mu}^{mn}\left[r\ \Bigg|\begin{array}{c}
[\mathfrak{c},\gamma]\\
{}[\mathfrak{d},\delta]
\end{array}\right]\nonumber \\
 & :=\textnormal{H}_{\nu\mu}^{mn}\left[r\ \Big|\begin{array}{ccc}
(\mathfrak{c}_{1},\gamma_{1}) & \cdots & (\mathfrak{c}_{\nu},\gamma_{\nu})\\
(\mathfrak{d}_{1},\delta_{1}) & \cdots & (\mathfrak{d}_{\mu},\delta_{\mu})
\end{array}\right]:=\frac{1}{2\pi\textnormal{i}}\int_{L}\mathcal{H}(z)r^{-z}dz.
                                                                                               \label{eq:H function}
\end{align}
In (\ref{eq:H function}), $L$ is the infinite contour which separates
all the poles in $P_{1}$ to the left and all the poles in $P_{2}$
to the right of $L$. Precisely,  we  choose $L$  as follows:
\begin{enumerate}
\item if $\omega>0$, then $L=L_{Ha}^{-}$, which is a left loop situated
in a horizantal strip (or left Hankel contour), runs from $-\infty+\textnormal{i}h_{1}$
to $\ell+\textnormal{i}h_{1}$, and then to $\ell+\textnormal{i}h_{2}$
and finally terminates at the point $-\infty+\textnormal{i}h_{2}$
with $-\infty<h_{1}<0<h_{2}<\infty$ and
\begin{align}
				\label{eq:ell condition}
\hat{\mathfrak{d}}_{0}<\ell<\hat{\mathfrak{c}}_{0},
\end{align}

\item if $\omega<0$, then $L=L_{Ha}^{+}$, which is a right loop situated
in a horizatal strip (or right Hankel contour), runs from $+\infty+\textnormal{i}h_{1}$
to $\ell+\textnormal{i}h_{1}$, and then to $\ell+\textnormal{i}h_{2}$
and finally terminates at the point $+\infty+\textnormal{i}h_{2}$
with $-\infty<h_{1}<0<h_{2}<\infty$ and (\ref{eq:ell condition}),
\item if $\omega=0$, then $L=L_{Ha}^{-}$ and $L=L_{Ha}^{+}$ and for $r\in(0,\eta)$
and $r\in(\eta,\infty)$ respectively.
\end{enumerate}

The following proposition shows that  integral (\ref{eq:H function}) is well defined and
 independent of the choice of $h_{1},h_{2}$, and $\ell\in(\hat{\mathfrak{d}}_{0},\hat{\mathfrak{c}}_{0})$.

\begin{prop}[{\cite[Theorem 1.2]{kilbas2004h}}]
										\label{prop:H function residue}
Assume (\ref{eq:P1 P2}) and choose the contour $L$ as above. Then Mellin-Barnes integral (\ref{eq:H function}) makes sense and it is an analytic
function of $r\in(0,\infty)$ and of $r\in(0,\eta)\cup(\eta,\infty)$
if $\omega\neq0$ and $\omega=0$ respectively. Furthermore,\end{prop}
\begin{enumerate}
\item if $\omega>0$, then
\begin{align}
					\label{eq:residue positive w}
\textnormal{H}_{\nu\mu}^{mn}(r)=\sum_{k=0}^{\infty}\textnormal{Res}_{z=\hat{\mathfrak{d}}_{k}}\left[\mathcal{H}(z)r^{-z}\right]
\end{align}

\item if $\omega<0$, then
\begin{align}
				\label{eq:residue negative w}
\textnormal{H}_{\nu\mu}^{mn}(r)=-\sum_{k=0}^{\infty}\textnormal{Res}_{z=\hat{\mathfrak{c}}_{k}}\left[\mathcal{H}(z)r^{-z}\right]
\end{align}

\item if $\omega=0$, then (\ref{eq:residue positive w}) and (\ref{eq:residue negative w})
hold for $r\in(0,\eta)$ and $r\in(\eta,\infty)$ respectively.
\end{enumerate}

\begin{proof}
(i) Let $\omega>0$ and $r\in(0,\infty)$. Choose $h_{1}<0$, $h_{2}>0$,
and $\ell\in\mathbb{R}$ so that  (\ref{eq:ell condition}) holds. Take a sufficiently
large $p\in\mathbb{Z}_{+}$ so that $\hat{\mathfrak{d}}_{p}<0$. Then
there exists a real number $M=M(p)>0$ such that
\begin{align}
						\label{eq:M}
-\hat{\mathfrak{d}}_{p}<M<-\hat{\mathfrak{d}}_{p+1}.
\end{align}
Define a closed rectangular contour $C^{M}$ which can be decomposed
into four lines
$$
C_{M}=L_{1}\cup L_{2}\cup L_{3}\cup L_{4}
$$
where
$$
L_{1}:=\left\{ z\in\mathbb{C}:\Re[z]=\ell,h_{1}\leq\Im[z]\leq h_{2}\right\} ,
$$
$$
L_{2}:=\left\{ z\in\mathbb{C}:\Re[z]=-M,h_{1}\leq\Im[z]\leq h_{2}\right\} ,
$$
$$
L_{3}:=\left\{ z\in\mathbb{C}:-M\leq\Re[z]\leq\ell,\Im[z]=h_{1}\right\} ,
$$
$$
L_{4}:=\left\{ z\in\mathbb{C}:-M\leq\Re[z]\leq\ell,\Im[z]=h_{2}\right\} .
$$
Note that $\mathcal{H}(z)r^{-z}$ is a meromorphic function on $z\in\mathbb{C}\setminus P_{1}\cup P_{2}$ and
\begin{align}
					\label{eq:abso 1}
\int_{C^{M}}\left|\mathcal{H}(z)r^{-z}\right||dz|<\infty.
\end{align}
Due to (\ref{eq:Stirling approx horizantal-}) ($i=1,2$)
\begin{align}
\left|\int_{-\infty}^{-M}\mathcal{H}(t+\textnormal{i}h_{i})r^{-t-\textnormal{i}h_{i}}dt\right| & \leq\int_{-\infty}^{-M}\left|\mathcal{H}(t+\textnormal{i}h_{i})r^{-t-\textnormal{i}h_{i}}\right|dt\nonumber \\
 & \apprle\int_{M}^{\infty}\left(\frac{e}{t}\right)^{\omega t}\left(\frac{r}{\eta}\right)^{-t}t^{\Lambda}dt<\infty,
                  \label{eq:abso 2}
\end{align}
and
\begin{align}
\left|\int_{L_{2}}\mathcal{H}(z)r^{-z}dz\right| & =\bigg|\int_{h_{1}}^{h_{2}}\mathcal{H}\left(-M+\textnormal{i}t\right)r^{M-\textnormal{i}t}dt\bigg|\nonumber \\
 & \leq\int_{h_{1}}^{h_{2}}\left|\mathcal{H}(-M+\textnormal{i}t)r^{M-\textnormal{i}t}\right|dt\nonumber \\
 & \apprle\left(\frac{e}{M}\right)^{\omega M}\left(\frac{r}{\eta}\right)^{M}M^{\Lambda}(h_{2}-h_{1})<\infty.\label{eq:abso 3}
\end{align}
By (\ref{eq:abso 1})-(\ref{eq:abso 3}),  integral (\ref{eq:H function})
absolutely converges and makes sense with $L=L_{Ha}^{-}$.

We next show (\ref{eq:residue positive w}). Note that (\ref{eq:abso 2})
and (\ref{eq:abso 3}) converge to 0 as $M\rightarrow\infty$ and
by the theory of residues,
$$
\frac{1}{2\pi\textnormal{i}}\int_{C_{M}}\mathcal{H}(z)r^{-z}dz=\sum_{k=0}^{p}\textnormal{Res}_{z=\hat{\mathfrak{d}}_{k}}\left[\mathcal{H}(z)r^{-z}\right].
$$
Thus our claim follows immediately if we prove the convergence of the
residue expansion. Let $q\in\mathbb{N}$ $(q>p)$ be arbitrarily given. Then
 we can choose a real number $N=N(q)>0$ satisfying (\ref{eq:M})
where $p$ and $M$ are replaced by $q$ and $N$ respectively. Set
$$
C_{M}':=L_{1}'\cup L_{2}\cup L_{3}'\cup L_{4}'
$$
where
$$
L_{1}':=\left\{ z\in\mathbb{C}:\Re[z]=-N,h_{1}\leq\Im[z]\leq h_{2}\right\} ,
$$
$$
L_{3}':=\left\{ z\in\mathbb{C}:-N\leq\Re[z]\leq-M,\Im[z]=h_{1}\right\} ,
$$
$$
L_{4}':=\left\{ z\in\mathbb{C}:-N\leq\Re[z]\leq-M,\Im[z]=h_{2}\right\} .
$$
Observe that
\begin{align}
				\label{eq:partial sum}
\frac{1}{2\pi\textnormal{i}}\int_{C_{M}'}\mathcal{H}(z)r^{-z}dz=\sum_{k=p+1}^{q}\textnormal{Res}_{z=\hat{\mathfrak{d}}_{k}}\left[\mathcal{H}(z)r^{-z}\right].
\end{align}
By replacing $M$ by $N$ in (\ref{eq:abso 2}) and (\ref{eq:abso 3}),
$$
\lim_{N\rightarrow\infty}\left|\int_{L_{1}'}\mathcal{H}(z)r^{-z}dz\right|=\lim_{N,M\rightarrow\infty}\left|\int_{L_{3}'\cup L_{4}'}\mathcal{H}(z)r^{-z}dz\right|=0,
$$
which implies
\begin{align*}
\lim_{p,q\rightarrow\infty}\sum_{k=p+1}^{q}\textnormal{Res}_{z=\hat{\mathfrak{d}}_{k}}\left[\mathcal{H}(z)r^{-z}\right] & =\lim_{N,M\rightarrow\infty}\frac{1}{2\pi\textnormal{i}}\int_{C_{M}'}\mathcal{H}(z)r^{-z}dz=0.
\end{align*}
Thus  (\ref{eq:residue positive w}) is proved.

(ii) The case $\omega<0$
is an analogue of the case $\omega>0$. By (\ref{eq:Stirling approx horizantal+}),
for sufficiently large $M>0$ ($i=1,2$)
\begin{align}
\left|\int_{M}^{\infty}\mathcal{H}(t+\textnormal{i}h_{i})r^{-t-\textnormal{i}h_{i}}dt\right| & \leq\int_{M}^{\infty}\left|\mathcal{H}(t+\textnormal{i}h_{i})r^{-t-\textnormal{i}h_{i}}\right|dt\nonumber \\
 & \apprle\int_{M}^{\infty}\left(\frac{e}{t}\right)^{-\omega t}\left(\frac{\eta}{r}\right)^{t}t^{\Lambda}dt<\infty,
                  \label{eq:abso 2-1}
\end{align}
and
\begin{align}
\left|\int_{L_{2}}\mathcal{H}(z)r^{-z}dz\right| & =\bigg|\int_{h_{1}}^{h_{2}}\mathcal{H}\left(M+\textnormal{i}t\right)r^{M-\textnormal{i}t}dt\bigg|\nonumber \\
 & \leq\int_{h_{1}}^{h_{2}}\left|\mathcal{H}(M+\textnormal{i}t)r^{M-\textnormal{i}t}\right|dt\nonumber \\
 & \apprle\left(\frac{e}{M}\right)^{-\omega M}\left(\frac{\eta}{r}\right)^{M}M^{\Lambda}(h_{2}-h_{1})<\infty.\label{eq:abso 3-1}
\end{align}
Note that both \eqref{eq:abso 2-1} and \eqref{eq:abso 3-1} converge to 0 as $M\rightarrow\infty$. Then we obtain our desired result by replacing $\hat{\mathfrak{d}}_{k}$
and $M$ in the proof of the case $\omega>0$ by $\hat{\mathfrak{c}}_{k}$
and $-M$ respectively.

(iii) Finally we consider $\omega=0$. Note
that (\ref{eq:abso 2}) and (\ref{eq:abso 3}) hold if $r<\eta$,
consequently (\ref{eq:residue positive w}) follows. If $r>\eta$,
then we have (\ref{eq:abso 2-1}) and (\ref{eq:abso 3-1}) which give
(\ref{eq:residue negative w}) immediately. The proposition is proved.
\end{proof}

In the remainder of this section, we assume
\begin{align}	
						\label{eq:parameter}
\alpha^{*}>0.
\end{align}
Under (\ref{eq:parameter}), we can choose a contour $L=L_{Br}$ which
is a vertical contour (or Bromwich contour) starting at the point
$\ell-\textnormal{i}\infty$ and terminating at the point $\ell+\textnormal{i}\infty$
where $\ell$ satisfies (\ref{eq:ell condition}).
Actually, $\textnormal{H}_{\nu\mu}^{mn}(r)$ does not depend on the choice
of $L$ due to the following proposition and  it is an analytic function
of $r\in(0,\infty)$ (it is a holomorphic function of $r\in\mathbb{C}$
in the sector $|\arg r|<\frac{\omega\pi}{2}$. See \cite[Theorem 1.2.(iii)]{kilbas2004h}).

\begin{prop}
										\label{prop:Bromwich=00003DHankel}
Under (\ref{eq:parameter}), Mellin-Barnes integral (\ref{eq:H function})
makes sense with $L=L_{Br}$. Furthermore, for $r\in(0,\infty)$
\begin{align}
					\label{eq:Br = positive}
\frac{1}{2\pi\textnormal{i}}\int_{L_{Br}}\mathcal{H}(z)r^{-z}dz=\frac{1}{2\pi\textnormal{i}}\int_{L_{Ha}^{-}}\mathcal{H}(z)r^{-z}dz
\end{align}
if $\omega>0$,
\begin{align}
					\label{eq:Br = negative}
\frac{1}{2\pi\textnormal{i}}\int_{L_{Br}}\mathcal{H}(z)r^{-z}dz=\frac{1}{2\pi\textnormal{i}}\int_{L_{Ha}^{+}}\mathcal{H}(z)r^{-z}dz
\end{align}
if $\omega<0$. If $\omega=0$ then (\ref{eq:Br = positive})
and (\ref{eq:Br = negative}) hold for $r<\eta$ and $r>\eta$
respectively.
\end{prop}

\begin{proof}
We only prove the case $\omega\geq0$. The proof of the other case is almost same.  Fix $r\in(0,\infty)$
and
$$
L=L_{Br}=\left\{ z\in\mathbb{C}:\Re[z]=\ell\right\},
$$
where $\ell$ satisfies (\ref{eq:ell condition}). Note that the relation
$$
\left|\mathcal{H}(\ell+\textnormal{i}t)r^{-\ell-\textnormal{i}t}\right|\sim e^{-\omega}\eta^{\ell}r^{-\ell}|t|^{\omega\ell+\Lambda}\exp\left\{ -\frac{|t|\alpha^{*}}{2}\pi\right\}
$$
holds as $|t|\rightarrow\infty$ uniformly in $\ell$ by (\ref{eq:Stirling approx vertical}).
Hence the contour integral
$$
\frac{1}{2\pi\textnormal{i}}\int_{L_{Br}}\mathcal{H}(z)r^{-z}dz
$$
 absolutely converges and makes sense. Due to Proposition \ref{prop:H function residue},
we only need to show that
\begin{align}
					\label{eq:Bromwich residue}
\frac{1}{2\pi\textnormal{i}}\int_{L_{Br}}\mathcal{H}(z)r^{-z}dz=\sum_{k=0}^{\infty}\textnormal{Res}_{z=\hat{\mathfrak{d}}_{k}}\left[\mathcal{H}(z)r^{-z}\right].
\end{align}
Given (sufficiently large) $p\in\mathbb{Z}_{+}$, we  take $M$
as in the proof of Proposition \ref{prop:H function residue}. Define
a closed rectangular contour $C^{M}$ which can be decomposed into
four lines
$$
C^{M}=L_{v}^{M}\cup L_{v}^{-M}\cup L_{h}^{-M}\cup L_{h}^{M}
$$
where
$$
L_{v}^{M}:=\left\{ z\in\mathbb{C}:\Re[z]=\ell,\left|\Im[z]\right|\leq M\right\} ,
$$
$$
L_{v}^{-M}:=\left\{ z\in\mathbb{C}:\Re[z]=-M,\left|\Im[z]\right|\leq M\right\} ,
$$
$$
L_{h}^{M}:=\left\{ z\in\mathbb{C}:-M\leq\Re[z]\leq\ell,\Im[z]=M\right\} ,
$$
$$
L_{h}^{-M}:=\left\{ z\in\mathbb{C}:-M\leq\Re[z]\leq\ell,\Im[z]=-M\right\} .
$$
Then by the theorem of residues,
\begin{align}
					\label{eq: sequence of residue}
\frac{1}{2\pi\textnormal{i}}\int_{C^{M}}\mathcal{H}(z)r^{-z}dz=\sum_{k=0}^{p}\textnormal{Res}_{z=\hat{\mathfrak{d}}_{k}}\left[\mathcal{H}(z)r^{-z}\right].
\end{align}
By (\ref{eq:Stirling approx horizantal-}), for any $h\in\mathbb{R}$,
$$
\left|\mathcal{H}(t+\textnormal{i}h)r^{-t-\textnormal{i}h}\right|\sim\left(\frac{e}{|t|}\right)^{\omega|t|}\left(\frac{r}{\eta}\right)^{|t|}|t|^{\Lambda}
$$
as $t\rightarrow-\infty$. Thus
\begin{align*}
\lim_{M\rightarrow\infty}\int_{L_{v}^{-M}}\left|\mathcal{H}(z)r^{-z}\right||dz| & =\lim_{M\rightarrow\infty}\int_{-M}^{M}\left|\mathcal{H}(-M+\textnormal{i}h)r^{M-\textnormal{i}h}\right|dh\\
 & \apprle\lim_{M\rightarrow\infty}\int_{-M}^{M}\left(\frac{e}{M}\right)^{\omega M}\left(\frac{r}{\eta}\right)^{M}M^{\Lambda}dt\\
 & \leq\lim_{M\rightarrow\infty}2\left(\frac{e}{M}\right)^{\omega M}\left(\frac{r}{\eta}\right)^{M}M^{\Lambda+1}=0.
\end{align*}
On the other hand, by (\ref{eq:Stirling approx vertical}),
\begin{align*}
\lim_{M\rightarrow\infty}\int_{L_{h}^{M}}\left|\mathcal{H}(z)r^{-z}\right||dz| & =\lim_{M\rightarrow\infty}\int_{\ell}^{-M}\left|\mathcal{H}(t+\textnormal{i}M)r^{-t-\textnormal{i}M}\right|dt\\
 & \apprle\lim_{M\rightarrow\infty}\int_{\ell}^{-M}e^{-\omega t}\left(\frac{\eta}{r}\right)^{t}M^{\omega t+\Lambda}\exp\left\{ -\frac{\alpha^{*}M\pi}{2}\right\} dt\\
 & =\lim_{M\rightarrow\infty}\exp\left\{ -\frac{\alpha^{*}M\pi}{2}\right\} M^{\Lambda}\int_{\ell}^{-M}\left(\frac{\eta M^{\omega}}{re^{\omega}}\right)^{t}dt\\
 & =\lim_{M\rightarrow\infty}\exp\left\{ -\frac{\alpha^{*}M\pi}{2}\right\} M^{\Lambda}\cdot\frac{\left(\frac{\eta M^{\omega}}{re^{\omega}}\right)^{-M}-\left(\frac{\eta M^{\omega}}{re^{\omega}}\right)^{\ell}}{\ln\eta-\ln r+\omega(\ln M-1)}=0
\end{align*}
since $\alpha^{*}>0$, $\omega\geq0$, and $r<\eta$. Similarly,
$$
\lim_{M\rightarrow\infty}\int_{L_{h}^{-M}}\left|\mathcal{H}(z)r^{-z}\right||dz|=0.
$$
Thus, by taking $p\rightarrow\infty$ in (\ref{eq: sequence of residue}),
\begin{align*}
\sum_{k=0}^{\infty}\textnormal{Res}_{z=\hat{\mathfrak{d}}_{k}}\left[\mathcal{H}(z)r^{-z}\right] & =\lim_{M\rightarrow\infty}\frac{1}{2\pi\textnormal{i}}\int_{C^{M}}\mathcal{H}(z)r^{-z}dz\\
 & =\lim_{M\rightarrow\infty}\frac{1}{2\pi\textnormal{i}}\left(\int_{L_{v}^{M}}+\int_{L_{v}^{-M}}+\int_{L_{h}^{M}}+\int_{L_{h}^{-M}}\right)\\
 & =\lim_{M\rightarrow\infty}\frac{1}{2\pi\textnormal{i}}\int_{L_{v}^{M}}\mathcal{H}(z)r^{-z}dz=\frac{1}{2\pi\textnormal{i}}\int_{L_{Br}}\mathcal{H}(z)r^{-z}dz.
\end{align*}
Therefore (\ref{eq:Bromwich residue}) holds, and the proposition is proved.
\end{proof}

\begin{prop}[{\cite[(2.2.2)]{kilbas2004h}}]
								\label{prop:derivatives of H(r)}
\begin{align*}
\frac{d}{dr}\Bigg\{ & \textnormal{H}_{\nu\mu}^{mn}\left[r\ \Bigg|\begin{array}{c}
[\mathfrak{c},\gamma]\\
{}[\mathfrak{d},\delta]
\end{array}\right]\Bigg\}=-r^{-1}\textnormal{H}_{\nu+1\mu+1}^{m+1\ n}\left[r\ \Bigg|\begin{array}{cc}
[\mathfrak{c},\gamma] & (0,1)\\
(1,1) & [\mathfrak{d},\delta]
\end{array}\right].
\end{align*}

\end{prop}

\begin{proof}
Observe that $\omega$ and $\alpha^{*}$ of $\textnormal{H}_{\nu\mu}^{mn}(r)$
and of
$$
\textnormal{H}_{\nu+1\mu+1}^{m+1\ n}\left[r\ \Bigg|\begin{array}{cc}
[\mathfrak{c},\gamma] & (0,1)\\
(1,1) & [\mathfrak{d},\delta]
\end{array}\right]
$$
are same, respectively. Hence $\textnormal{H}_{\nu+1\mu+1}^{m+1\ n}(r)$
is well-defined with the same contour used for $\textnormal{H}_{\nu\mu}^{mn}(r)$. Therefore the proposition easily follows from
(\ref{eq:zGamma(z)=Gamma(z+1)}), (\ref{eq:Stirling approx horizantal+})-(\ref{eq:Stirling approx vertical}), and the definition of the
Fox H function.
\end{proof}

\subsection{Algebraic asymptotic expansions of $\textnormal{H}_{\nu\mu}^{mn}(r)$
near zero and at infinity.}

Proposition \ref{prop:H function residue} gives explicit power and
power-logarithmic expansion of $\textnormal{H}_{\nu\mu}^{mn}(r)$
near zero for $\omega\geq0$ and at infinity for $\omega\leq0$. For
the reverse case (i.e.) near zero for $\omega<0$ and at infinity
for $\omega>0$, the following asymptotic expansions hold.
\begin{prop}
								\label{prop:asymptotic expansion}
Suppose (\ref{eq:parameter}) holds. Then for sufficiently large $p\in\mathbb{Z}_{+}$, if $\omega>0$
\begin{align}
					\label{eq:residue approxi large x}
\textnormal{H}_{\nu\mu}^{mn}(r)=\sum_{k=0}^{p}\textnormal{Res}_{z=\hat{\mathfrak{c}}_{k}}\left[\mathcal{H}(r)r^{-z}\right]+O(r^{-M}),\quad\hat{\mathfrak{c}}_{p}<M<\hat{\mathfrak{c}}_{p+1}
\end{align}
as $r\rightarrow\infty$, and if $\omega<0$
\begin{align}
				\label{eq:residue approxi small x}
\textnormal{H}_{\nu\mu}^{mn}(r)=-\sum_{k=0}^{p}\textnormal{Res}_{z=\hat{\mathfrak{d}}_{k}}\left[\mathcal{H}(r)r^{-z}\right]+O(r^{M}),\quad-\hat{\mathfrak{d}}_{p}<M<-\hat{\mathfrak{d}}_{p+1}
\end{align}
as $r\rightarrow0$.
\end{prop}

\begin{proof}
Braaksma \cite{braaksma1964asymptotic} proved (\ref{eq:residue approxi large x}) for $\omega\geq0$.  We follow Braaksma's method to prove (\ref{eq:residue approxi small x}) for the case $\omega<0$.

Let $\omega<0$. Given (sufficiently large) $p\in\mathbb{Z}_{+}$,
take a constant $M>0$ satisfying (\ref{eq:M}). Fix $h>0$ and
let $L_{Ha}^{-M}$ be a right Hankel contour surrounding $L_{Ha}^{+}$.
Precisely, $L_{Ha}^{-M}$ is a right loop situated in a horizantal
strip runs from $\infty-\textnormal{i}h$ to $-M-\textnormal{i}h$
and then to $-M+\textnormal{i}h$ and finally terminating at the point
$\infty+\textnormal{i}h$. Let us denote
by $C$ a closed rectangular contour which encircles $\hat{\mathfrak{d}}_{0},\ldots\hat{\mathfrak{d}}_{p}$
and satisfies
$$
\int_{L_{Ha}^{+}}\mathcal{H}(z)r^{-z}dz+\int_{C}\mathcal{H}(z)r^{-z}dz=\int_{L_{Ha}^{-M}}\mathcal{H}(z)r^{-z}dz.
$$
Then
\begin{align}
\textnormal{H}_{\nu\mu}^{mn}(r) & =\frac{1}{2\pi\textnormal{i}}\int_{L_{Ha}^{+}}\mathcal{H}(z)r^{-z}dz\nonumber \\
 & =\frac{1}{2\pi\textnormal{i}}\left(-\int_{C}\mathcal{H}(z)r^{-z}dz+\int_{L_{Ha}^{-M}}\mathcal{H}(z)r^{-z}dz\right)\nonumber \\
 & =-\sum_{k=0}^{p}\textnormal{Res}_{z=\hat{\mathfrak{d}}_{k}}\left[\mathcal{H}(z)r^{-z}\right]+\frac{1}{2\pi\textnormal{i}}\int_{L_{Ha}^{-M}}\mathcal{H}(z)r^{-z}dz.\label{eq:exact approximation small r}
\end{align}
Using (\ref{eq:Stirling approx horizantal+}), (\ref{eq:Stirling approx horizantal-}),
and (\ref{eq:parameter}), and modifying the proof of Proposition
\ref{prop:Bromwich=00003DHankel},
\begin{align*}
\frac{1}{2\pi\textnormal{i}}\int_{L_{Ha}^{-M}}\mathcal{H}(z)r^{-z}dz & =\sum_{k=0}^{p}\textnormal{Res}_{z=\hat{\mathfrak{d}}_{k}}\left[\mathcal{H}(z)r^{-z}\right]+\sum_{k=0}^{\infty}\textnormal{Res}_{z=\hat{\mathfrak{c}}_{k}}\left[\mathcal{H}(z)r^{-z}\right]\\
 & =\frac{1}{2\pi\textnormal{i}}\int_{L_{Br}^{-M}}\mathcal{H}(z)r^{-z}dz,
\end{align*}
where $L_{Br}^{-M}$ is a Bromwich contour starting at the point $-M-\textnormal{i}\infty$
and terminating at the point $-M+\textnormal{i}\infty$.

We estimate
the upper bound of contour integral along $L_{Br}^{-M}$. Recall (\ref{eq:Stirling approx vertical}):
$$
\left|\mathcal{H}(-M+\textnormal{i}t)r^{M-\textnormal{i}t}\right|\sim e^{\omega M}\eta^{-M}r^{M}|t|^{-\omega M+\Lambda}\exp\left\{ -\frac{|t|\alpha^{*}}{2}\pi\right\}
$$
as $t\rightarrow\infty$. Thus, for $r\leq 1$,
\begin{align*}
\left|\frac{1}{2\pi\textnormal{i}}\int_{L_{Br}^{-M}}\mathcal{H}(z)r^{-z}dz\right| & \leq\frac{1}{2\pi}\int_{-\infty}^{\infty}\left|\mathcal{H}(-M+\textnormal{i}t)r^{M-it}\right|dt\\
 & \apprle r^{M}\left(1+\frac{e^{\omega M}}{\eta^{M}}\int_{|t|\geq1}|t|^{-\omega M+\Lambda}\exp\left\{ -\frac{|t|\alpha^{*}}{2}\pi\right\} dt\right)\\
 & \apprle r^{M}.
\end{align*}
The proposition is proved.
\end{proof}

If $\omega>0$ and $n=0$ (i.e.) $P_{2}=\emptyset$,  we have the following exponentially asymptotic behavior of $\textnormal{H}_{\nu\mu}^{m0}(r)$  (see \cite[(1.7.13)]{kilbas2004h}). For the proof we refer the reader to \cite[Theorem 4]{braaksma1964asymptotic}.

\begin{prop}
									\label{prop:exponential decay}
Assume \eqref{eq:parameter} and $\omega>0$. Then
\begin{align*}
\textnormal{H}_{\nu\mu}^{m0}(r)=O\left(r^{(\Lambda+\frac{1}{2})/\omega}\exp\left\{\cos\left(\frac{\alpha^{*}+\sum_{j=m+1}^{\mu}\mathfrak{\delta}_{j}}{\omega}\pi\right)\omega \left(\frac{r}{\eta}\right)^{1/\omega}\right\}\right)
\end{align*}
as $r\to\infty$.
\end{prop}

\section{\label{sec:Auxiliary-results}Asymptotic estimates of the Fox H function}

Throughout this section we fix
$$
d\in\mathbb{N}, \quad \alpha\in(0,2), \quad \beta\in(0,\infty).
$$
For $\gamma\in [0,\infty), \sigma\in \mathbb{R}$ and $z\in \mathbb{C}$ define
$$
\mathcal{H}_{\sigma,\gamma}(z):=\frac{\Gamma(\frac{d}{2}+\gamma+\beta z)\Gamma(1+z)\Gamma(-z)}{\Gamma(-\gamma-\beta z)\Gamma(1-\sigma+\alpha z)}.
$$
Observe that
$$
\alpha^{*}=2-\alpha,\quad\Lambda=\frac{d}{2}+2\gamma+\sigma-\frac{1}{2},\quad\omega=2\beta-\alpha,\quad\eta=\alpha^{-\alpha}\beta^{2\beta}.
$$

For each $k\in\mathbb{Z}_{+}$ we write
$$
\mathfrak{c}_{1,k}:=k,\quad\mathfrak{d}_{1,k}:=-\frac{\frac{d}{2}+\gamma+k}{\beta},\quad\mathfrak{d}_{2,k}:=-1-k.
$$
Obviously, (\ref{eq:P1 P2}) holds, and thus $\mathfrak{d}_{1,k}$,
$\mathfrak{d}_{2,k}$ and $\mathfrak{c}_{1,k}$ constitute
$P_{1}$ and $P_{2}$ respectively.

\begin{rem}
					\label{rem:removable singularity}
(i) Note that $\mathcal{H}_{\sigma,\gamma}(z)$
has removable singularities at $z=0$ and at $z=-1$ if $\gamma=0$
and $\gamma=\beta$, respectively. Indeed, by (\ref{eq:zGamma(z)=Gamma(z+1)}),
$$
\mathcal{H}_{\sigma,0}(z)=\frac{\beta\Gamma(\frac{d}{2}+\beta z)\Gamma(1+z)\Gamma(1-z)}{\Gamma(1-\beta z)\Gamma(1-\sigma+\alpha z)}.
$$
Therefore, $\mathcal{H}_{\sigma,0}(z)$ has a removable singularity at $z=0$. Similarly,
$$
\mathcal{H}_{\sigma,\beta}(z)=-\frac{\beta\Gamma(\frac{d}{2}+\beta+\beta z)\Gamma(2+z)\Gamma(-z)}{\Gamma(1-\beta-\beta z)\Gamma(1-\sigma+\alpha z)}.
$$
Thus, $\mathcal{H}_{\sigma,\beta}(z)$ has a removable singularity at $z=-1$.

(ii) Assume that $\beta\in\mathbb{N}$ and $\gamma=0$, then by \eqref{eq:gamma property},
\begin{align}
					\label{eq:gamma=0 beta natural}
\mathcal{H}_{\sigma,0}(z)=(2\pi)^{\frac{\beta-1}{2}}\frac{\Gamma(\frac{d}{2}+\beta z)\Gamma(1+z)}{\prod_{k=1}^{\beta-1}\Gamma(\frac{k}{\beta}-z)\Gamma(1-\sigma+\alpha z)}\beta^{\beta z+\frac{1}{2}}.
\end{align}
Hence $P_{2}=\emptyset$.

(iii) If $\alpha=1$ and $\sigma=0$, then
\begin{align}
					\label{eq:alpha=1 sigma=0}
\mathcal{H}_{0,\gamma}(z):=\frac{\Gamma(\frac{d}{2}+\gamma+\beta z)\Gamma(-z)}{\Gamma(-\gamma-\beta z)}.
\end{align}
Thus, $\mathfrak{d}_{2,k}=0$ for all $k\in\mathbb{Z}_{+}$.
\end{rem}

For $r\in(0,\infty)$, we define
\begin{align}
\mathbb{H}_{\sigma,\gamma}(r) & :=\textnormal{H}_{23}^{21}\left[r\Big|\begin{array}{ccc}
(1,1) & (1-\sigma,\alpha)\\
(\frac{d}{2}+\gamma,\beta) & (1,1) & (1+\gamma,\beta)
\end{array}\right]\nonumber \\
 & =\frac{1}{2\pi\textnormal{i}}\int_{L}\mathcal{H}_{\sigma,\gamma}(z)r^{-z}dz.
 \label{eq:kernel representation}
\end{align}
Here
$$
L=L_{Br}=\left\{ z\in\mathbb{C}:\Re[z]=\ell_{0}\right\}
$$
and  $\ell_{0}$ is chosen to satisfy (\ref{eq:ell condition}):
$$
\max\left(-1,-\frac{\gamma}{\beta}-\frac{d}{2\beta}\right)<\ell_{0}<0
$$
if $\gamma\notin\{0,\beta\}$,
$$
\max\left(-1,-\frac{d}{2\beta}\right)<\ell_{0}<1
$$
if $\gamma=0$,
$$
\max\left(-2,-1-\frac{d}{2\beta}\right)<\ell_{0}<0
$$
if $\gamma=\beta$. If $\alpha=1$ and $\sigma=0$, then we take $\ell_{0}$ such that
$$
-\frac{\gamma}{\beta}-\frac{d}{2\beta}<\ell_{0}<0.
$$
Since (\ref{eq:parameter}) holds (i.e. $\alpha<2$),
by Propositions \ref{prop:H function residue} and \ref{prop:Bromwich=00003DHankel},
the value of $\mathbb{H}_{\sigma,\gamma}(r)$ is independent of the
choice of $\ell_{0}$ as long as it is chosen as  above.

By Proposition \ref{prop:asymptotic expansion}, we obtain the asymptotic
behaviors  of $\mathbb{H}_{\sigma,\gamma}(r)$ at infinity.
\begin{lem}
					\label{lem:H large r}
It holds that
$$
\mathbb{H}_{\sigma,\gamma}(r)= -\frac{\Gamma(\frac{d}{2}+\gamma)}{\Gamma(-\gamma)\Gamma(1-\sigma)}+O(r^{-1})
$$
for $r\geq 1$. In particular, if $\gamma\in\mathbb{Z}_{+}$
or $\sigma\in\mathbb{N}$, then
$$
\mathbb{H}_{\sigma,\gamma}(r)= \frac{\Gamma(\frac{d}{2}+\gamma+\beta)}{\Gamma(-\gamma-\beta)\Gamma(1-\sigma+\alpha)}r^{-1}+O(r^{-2})
$$
for $r\geq 1$. If $\beta\in{\mathbb{N}}$ and $\gamma=0$, then there exists a constant $c=c(d, \alpha,\beta,\sigma)>0$ such that
$$
\mathbb{H}_{\sigma,0}(r)=O\left(\exp\left\{-cr^{1/(2\beta-\alpha)}\right\}\right)
$$
as $r\rightarrow\infty$.
\end{lem}

\begin{proof}
Observe that $\mathcal{H}_{\sigma,\gamma}(z)$ has simple poles at
$z=\hat{\mathfrak{c}}_{k}=\mathfrak{c}_{1,k}=k$ for all $k\in\mathbb{Z}_{+}$. By (\ref{eq:residue approxi large x}),
for sufficiently large $p\in\mathbb{Z}_{+}$,
\begin{align*}
\mathbb{H}_{\sigma,\gamma}(r) & =\sum_{k=0}^{1}\textnormal{Res}_{z=\hat{\mathfrak{c}}_{k}}\left[\mathcal{H}_{\sigma,\gamma}(z)r^{-z}\right]+\sum_{k=2}^{p}\textnormal{Res}_{z=\hat{\mathfrak{c}}_{k}}\left[\mathcal{H}_{\sigma,\gamma}(z)r^{-z}\right]+O(r^{-\hat{\mathfrak{c}}_{p}})\nonumber \\
 & =\sum_{k=0}^{1}\textnormal{Res}_{z=k}\left[\mathcal{H}_{\sigma,\gamma}(z)\right]r^{-k}+O(r^{-2})
\end{align*}
for $r\geq1$. Additionally, if $\gamma\in\mathbb{Z}_{+}$ or $\sigma\in\mathbb{N}$,
then
\begin{eqnarray*}
\textnormal{Res}_{z=0}\left[\mathcal{H}_{\sigma,\gamma}(z)\right] & = & \lim_{z\rightarrow0}\left(\frac{z\Gamma(\frac{d}{2}+\gamma+\beta z)\Gamma(1+z)\Gamma(-z)}{\Gamma(-\gamma-\beta z)\Gamma(1-\sigma+\alpha z)}\right)\\
 & = & -\lim_{z\rightarrow0}\left(\frac{\Gamma(\frac{d}{2}+\gamma+\beta z)\Gamma(1+z)\Gamma(1-z)}{\Gamma(-\gamma-\beta z)\Gamma(1-\sigma+\alpha z)}\right)=0.
\end{eqnarray*}
Hence
$$
\mathbb{H}_{\sigma,\gamma}(r)=\textnormal{Res}_{z=1}[\mathcal{H}_{\sigma,\gamma}(z)]r^{-1}+O(r^{-2})=\frac{\Gamma(\frac{d}{2}+\gamma+\beta)}{\Gamma(-\gamma-\beta)\Gamma(1-\sigma+\alpha)}r^{-1}+O(r^{-2}).
$$
Finally, assume $\beta\in\mathbb{N}$ and $\gamma=0$. By \eqref{eq:gamma=0 beta natural},
\begin{align*}
&\mathbb{H}_{\sigma,0}(r)=\textnormal{H}_{23}^{21}\left[r\Big|\begin{array}{ccc}
(1,1) & (1-\sigma,\alpha)\\
(\frac{d}{2},\beta) & (1,1) & (1,\beta)
\end{array}\right]
\\ &= \frac{(2\pi)^{\frac{\beta-1}{2}}\beta^{1/2}}{2\pi\textnormal{i}}\int_{L}\frac{\Gamma(\frac{d}{2}+\beta z)\Gamma(1+z)}{\prod_{k=1}^{\beta-1}\Gamma(\frac{k}{\beta}-z)\Gamma(1-\sigma+\alpha z)}\left(\frac{r}{\beta^{\beta}}\right)^{-z}dz
\\ &= (2\pi)^{\frac{\beta-1}{2}}\beta^{1/2} \textnormal{H}_{1\beta+1}^{20}\left[\frac{r}{\beta^{\beta}}\Big|\begin{array}{ccccc}
(1-\sigma,\alpha)\\
(\frac{d}{2},\beta) & (1,1) & (1+\frac{1}{\beta},1) & \cdots & (1+\frac{\beta-1}{\beta},1)
\end{array}\right].
\end{align*}
For the above Fox H function, we have
$$
\alpha^{*}=2-\alpha,\quad \Lambda=\frac{d}{2}+\sigma+2\beta-\frac{3}{2},\quad \omega=2\beta-\alpha,\quad \eta=\alpha^{-\alpha} \beta^\beta.
$$
Note that
$$
\frac{1}{2}<\frac{\beta+1-\alpha}{2\beta-\alpha}\leq 1,\quad \quad  2\beta-\alpha \geq 2-\alpha > 0.
$$
Hence by Proposition \ref{prop:exponential decay},
\begin{align*}
&\textnormal{H}_{1\beta+1}^{20}\left(\frac{r}{\beta^{\beta}}\right)
\\ &=O\left(\left(\frac{r}{\beta^{\beta}}\right)^{(\Lambda+\frac{1}{2})/2\beta-\alpha}\exp\left\{\cos\left(\frac{\beta+1-\alpha}{2\beta-\alpha}\pi\right)(2\beta-\alpha)\left(\frac{r}{\eta\beta^{\beta}}\right)^{1/(2\beta-\alpha)}\right\}\right)
\\ &=O\left(\exp\left\{-cr^{1/(2\beta-\alpha)}\right\}\right).
\end{align*}
The lemma is proved.
\end{proof}

Next we consider the asymptotic behavior of $\mathbb{H}_{\sigma,\gamma}(r)$
at zero.
\begin{lem}
						\label{lem:H small r}
It holds that
$$
\mathbb{H}_{\sigma,\gamma}(r)\sim\begin{cases}
r^{\frac{d+2\gamma}{2\beta}} & :\gamma<\beta-\frac{d}{2}\\
r|\ln r| & :\gamma=\beta-\frac{d}{2}\\
r & :\gamma>\beta-\frac{d}{2}
\end{cases}
$$
as $r\rightarrow0$. Additionally, if $\gamma-\beta\in\mathbb{Z}_{+}$
or $\sigma+\alpha\in\mathbb{N}$, then
$$
\mathbb{H}_{\sigma,\gamma}(r)\sim\begin{cases}
r^{\frac{d+2\gamma}{2\beta}} & :\gamma<2\beta-\frac{d}{2}\\
r^{2}|\ln r| & :\gamma=2\beta-\frac{d}{2}\\
r^{2} & :\gamma>2\beta-\frac{d}{2}
\end{cases}
$$
as $r\rightarrow0$. If $\alpha=1$ and $\sigma=0$, then
$$
\mathbb{H}_{0,\gamma}(r)\sim r^{\frac{d+2\gamma}{2\beta}}
$$
as $r\rightarrow 0$.
\end{lem}

\begin{proof}
Due to (\ref{eq:residue approxi small x}), it is sufficient to compare
the order of residues among $z=\mathfrak{d}_{1,0},\mathfrak{d}_{1,1},\mathfrak{d}_{2,0},$
and $\mathfrak{d}_{2,1}.$

First, let $\gamma\neq\beta-\frac{d}{2}$. Then $\mathcal{H}_{\sigma,\gamma}(z)$
has a simple pole at $z=\max\{\mathfrak{d}_{1,0},\mathfrak{d}_{2,0}\}$.
If $\gamma>\beta-\frac{d}{2}$, then $\mathfrak{d}_{1,0}<\mathfrak{d}_{2,0}$
so by (\ref{eq:residue approxi small x})
\begin{align*}
\mathbb{H}_{\sigma,\gamma}(r)=\sum_{k=0}^{1}\textnormal{Res}_{z=\mathfrak{d}_{2,k}}\left[\mathcal{H}_{\sigma,\gamma}(z)r^{-z}\right]+O(r^{-\mathfrak{d}_{2,1}}) \sim r
\end{align*}
as $r\rightarrow0$. Similarly, if $\gamma<\beta-\frac{d}{2}$, then
$$
\mathbb{H}_{\sigma,\gamma}(r)=\textnormal{Res}_{z=\mathfrak{d}_{1,0}}\left[\mathcal{H}_{\sigma,\gamma}(z)r^{-z}\right]+O(r^{-\mathfrak{d}_{1,0}})\sim r^{\frac{d+2\gamma}{2\beta}}
$$
as $r\rightarrow0$.

Next, assume $\gamma=\beta-\frac{d}{2}$ (i.e. $\mathfrak{d}_{1,0}=\mathfrak{d}_{2,0}$).
Then $\mathcal{H}_{\sigma,\gamma}(z)$ has a pole of order 2 at $z=\mathfrak{\hat{\mathfrak{d}}}_{0}=\mathfrak{d}_{1,0}=\mathfrak{d}_{2,0}$
so that
\begin{align*}
\textnormal{Res}_{z=\hat{\mathfrak{d}}_{0}} & \left[\mathcal{H}_{\sigma,\gamma}(z)r^{-z}\right]\\
 & =\lim_{z\rightarrow-1}\frac{d}{dz}\left((z+1)^{2}\mathcal{H}_{\sigma,\gamma}(z)r^{-z}\right)\\
 & =\left(\textnormal{Res}_{z=-1}\left[\mathcal{H}_{\sigma,\gamma}(z)\right]+\frac{|\ln r|}{\Gamma(\frac{d}{2})\Gamma(1-\sigma-\alpha)}\right)r.
\end{align*}
Then by (\ref{eq:residue approxi small x}), we obtain the first desired
result.

Now we assume $\gamma-\beta\in\mathbb{Z}_{+}$ or $\sigma+\alpha\in\mathbb{N}$.
Then one can easily see that $\textnormal{Res}_{z=\mathfrak{d}_{2,0}}\left[\mathcal{H}_{\sigma,\gamma}(z)r^{-z}\right]=0$.
Hence it remains to compare the order of residues at $z=\mathfrak{d}_{1,0},\mathfrak{d}_{1,1},$
and $\mathfrak{d}_{2,1}.$ Following the same argument of the
above, we obtain the additional result.

Finally, we assume $\alpha=1$ and $\sigma=0$. Recall \eqref{eq:alpha=1 sigma=0}, and note $\mathfrak{d}_{2,k}=0$ for all $k\in\mathbb{Z}_{+}$ which implies
$$
\mathbb{H}_{0,\gamma}(r)=\sum_{k=0}^{\infty}\textnormal{Res}_{z=\mathfrak{d}_{1,k}}\left[\mathcal{H}_{0,\gamma}(z)r^{-z}\right]\sim r^{\frac{d+2\gamma}{2\beta}}
$$
as $r\rightarrow 0$. The lemma is proved.
\end{proof}

For each $q\in\mathbb{Z}_{+}$ we define
\begin{align}
				\label{eq:cH^q}
\mathcal{H}_{\sigma,\gamma}^{(q)}(z):=\mathcal{H}_{\sigma,\gamma}(z)\left\{ \frac{\Gamma(1+z)}{\Gamma(z)}\right\} ^{q}
\end{align}
and
\begin{align}
					\label{eq:H^q}
\mathbb{H}_{\sigma,\gamma}^{(q)}(r):=\frac{1}{2\pi\textnormal{i}}\int_{L}\mathcal{H}_{\sigma,\gamma}^{(q)}(z)r^{-z}dz.
\end{align}
Note that $\mathbb{H}_{\sigma,\gamma}^{(0)}(r)=\mathbb{H}_{\sigma,\gamma}(r)$
and by Proposition \ref{prop:derivatives of H(r)},
$$
\frac{d}{dr}\mathbb{H}_{\sigma,\gamma}^{(q)}(r)=\mathbb{H}_{\sigma,\gamma}^{(q+1)}(r)
$$
holds and (\ref{eq:H^q}) is well-defined for each $q\in\mathbb{Z}_{+}$.
By (\ref{eq:zGamma(z)=Gamma(z+1)})
$$
\mathcal{H}_{\sigma,\gamma}^{(q)}(z)=-\frac{\Gamma(\frac{d}{2}+\gamma+\beta z)\Gamma(1-z)\Gamma(1+z)}{\Gamma(-\gamma-\beta z)\Gamma(1-\sigma+\alpha z)}z^{q-1},
$$
and thus  $\mathcal{H}_{\sigma,\gamma}^{(q)}(z)$ does not have a pole at
$z=0$ if $q\geq1$. Furthermore, $\mathcal{H}_{\sigma,\gamma}^{(q)}$
has a pole at $z=-k-1$ of order at most 2 for each $k\in\mathbb{Z}_{+}$.
Let us denote by
$$
\mathfrak{c}_{1,k}:=k+1,\quad\mathfrak{d}_{1,k}:=-\frac{\frac{d}{2}+\gamma+k}{\beta},\quad\mathfrak{d}_{2,k}:=-1-k
$$
the new elements of $P_{1}$ and $P_{2}$ for $\mathbb{H}_{\sigma,\gamma}^{(q)}(r)$.
Then due to (\ref{eq:residue approxi large x}) again, we obtain an
analogue of Lemma \ref{lem:H large r}.
\begin{lem}
								\label{lem:H^q large}
Let $q\in\mathbb{N}$. It holds that
$$
\mathbb{H}_{\sigma,\gamma}^{(q)}(r)\sim r^{-1}
$$
as $r\rightarrow\infty$. Additionally, if $\beta\in\mathbb{N}$ and $\gamma=0$, then there exists a constant $c=c(d, \alpha,\beta,\sigma, |q|)$ such that
$$
\mathbb{H}_{\sigma,0}^{(q)}(r)=O\left(\exp\left\{-cr^{1/(2\beta-\alpha)}\right\}\right)
$$
as $r\rightarrow\infty$.
\end{lem}

\begin{proof}
The proof is similar to the one of Lemma \ref{lem:H large r}. The
only difference is $\hat{\mathfrak{c}}_{k}=\mathfrak{c}_{1,k}=k+1$ which implies
$$
\textnormal{Res}_{z=\hat{\mathfrak{c}}_{k}}\left[\mathcal{H}_{\sigma,\gamma}^{(q)}(z)r^{-z}\right]=\frac{(-1)^{k}\cdot\Gamma(\frac{d}{2}+\gamma+\beta+\beta k)}{\Gamma(-\gamma-\beta-\beta k)\Gamma(1-\sigma+\alpha+\alpha k)}(k+1)^{q}r^{-k-1}
$$
for each $k\in\mathbb{Z}_{+}$. The lemma is proved.
\end{proof}

Lastly, we present another result which is necessary to obtain the
upper estimates of classical derivatives of $p(t,x)$. Recall $\omega=2\beta-\alpha$.
Let $\kappa_{1}$, $\kappa_{2}$, $\hat{\kappa}_{1}$, and $\hat{\kappa}_{2}$
denote constants
$$
\kappa_{1}:=\textnormal{Res}_{z=\mathfrak{d}_{1,0}}\left[\mathcal{H}_{\sigma,\gamma}(z)\right],\quad \kappa_{2}:=\textnormal{Res}_{z=\mathfrak{d}_{2,0}}\left[\mathcal{H}_{\sigma,\gamma}(z)\right]
$$
$$
\hat{\kappa}_{1}:=\lim_{z\rightarrow \mathfrak{d}_{1,0}}(z-\mathfrak{d}_{1,0})^{2}\mathcal{H}_{\sigma,\gamma}(z),\quad \hat{\kappa}_{2}:=\lim_{z\rightarrow \mathfrak{d}_{2,0}}(z-\mathfrak{d}_{2,0})^{2}\mathcal{H}_{\sigma,\gamma}(z)
$$
which are independent of $q\geq 1$. Note that $\hat{\kappa}_{2}=0$ if $\sigma+\alpha \in \mathbb{N}$.

\begin{lem}
						\label{lem:H^q small}
Let $q\in\mathbb{Z}_{+}$ and $\gamma\in[0,\infty)$.

(i) If $\gamma\notin\{\beta-\frac{d}{2},\beta-\frac{d}{2}-1,2\beta-\frac{d}{2} \}$, then
\begin{align*}
\mathbb{H}_{\sigma,\gamma}^{(q)}(r) & =\frac{\omega}{|\omega|}\cdot \left(-\frac{d+2\gamma}{2\beta}\right)^{q}\cdot\kappa_{1}r^{\frac{d+2\gamma}{2\beta}}+O(r^{\min\left(1,\frac{d+2\gamma+2}{2\beta}\right)}),\quad \left(\frac{0}{0}:=1\right)
\end{align*}
as $r\rightarrow0$. Additionally, if $\gamma=\beta$ or $\sigma+\alpha\in\mathbb{N}$,
then
$$
\mathbb{H}_{\sigma,\gamma}^{(q)}(r)=\frac{\omega}{|\omega|}\cdot\left(-\frac{d+2\gamma}{2\beta}\right)^{q}\cdot\kappa_{1}r^{\frac{d+2\gamma}{2\beta}}+\begin{cases}
O(r^{\min\left(2,\frac{d+2\gamma+2}{2\beta}\right)}) & :\gamma\neq2\beta-\frac{d}{2}-1\\
O(r^{2}|\ln r|) & :\gamma=2\beta-\frac{d}{2}-1
\end{cases}
$$
as $r\rightarrow0$. If $\alpha=1$ and $\sigma=0$, then

$$
\mathbb{H}_{0,\gamma}^{(q)}(r)=\frac{\omega}{|\omega|}\cdot\left(-\frac{d+2\gamma}{2\beta}\right)^{q}\cdot\kappa_{1}r^{\frac{d+2\gamma}{2\beta}}+O(r^{\frac{d+2\gamma+2}{2\beta}})
$$
as $r\rightarrow0$.

(ii)  If $\gamma=\beta-\frac{d}{2}$, then
\begin{align*}
\mathbb{H}_{\sigma,\gamma}^{(q)}(r) =\frac{\omega}{|\omega|}\cdot\left(-\frac{d+2\gamma}{2\beta}\right)^{q}\cdot (\hat{\kappa}_{2}\ln r+\kappa_{2}) r+O(r)
\end{align*}
as $r\rightarrow0$. Additionally, if $\sigma+\alpha\in\mathbb{N}$, then
\begin{align*}
\mathbb{H}_{\sigma,\gamma}^{(q)}(r) =-\frac{\omega}{|\omega|}\cdot\left(-\frac{d+2\gamma}{2\beta}\right)^{q}\cdot \kappa_{2} r+
\begin{cases}
O(r^{\min(2,\frac{d+2\gamma+2}{2\beta})}) &: \beta\neq 1 \\
O(r^{2}|\ln r|) &:\beta =1 
\end{cases}
\end{align*}
as $r\rightarrow 0$.

(iii) If $\gamma=\beta-\frac{d}{2}-1$, then
$$
\mathbb{H}_{\sigma,\gamma}^{(q)}(r)=\frac{\omega}{|\omega|}\cdot\left(-\frac{d+2\gamma}{2\beta}\right)^{q}\cdot\kappa_{1}r^{\frac{d+2\gamma}{2\beta}}+O(r|\ln r|)
$$
as $r\rightarrow0$. Additionally, if $\sigma+\alpha\in\mathbb{N}$, then
$$
\mathbb{H}_{\sigma,\gamma}^{(q)}(r)=\frac{\omega}{|\omega|}\cdot\left(-\frac{d+2\gamma}{2\beta}\right)^{q}\cdot\kappa_{1}r^{\frac{d+2\gamma}{2\beta}}+O(r)
$$
as $r\to 0$.

(iv) If $\gamma=2\beta-\frac{d}{2}$, then
$$
\mathbb{H}_{\sigma,\gamma}^{(q)}(r)=\frac{\omega}{|\omega|}\cdot\left(-\frac{d+2\gamma}{2\beta}\right)^{q}\cdot \left(\hat{\kappa}_{1}\ln r+\kappa_{1}\right)r^{2}+O(r)
$$
as $r\rightarrow 0$. Additionally, if $\sigma+\alpha\in\mathbb{N}$, then
$$
\mathbb{H}_{\sigma,\gamma}^{(q)}(r)=\frac{\omega}{|\omega|}\cdot\left(-\frac{d+2\gamma}{2\beta}\right)^{q}\cdot \left(\hat{\kappa}_{1}\ln r+\kappa_{1}\right)r^{2}+O(r^{2})
$$
as $r\to0$.
\end{lem}

\begin{proof}
Due to Proposition \ref{prop:H function residue}, one can easily
see our assertions hold if $\omega\geq0$. Hence, we assume $\omega<0$. Recall
(\ref{eq:residue approxi small x}) and take $M\geq2$ satisfying (\ref{eq:M}).

(i) If $\gamma\neq\beta-\frac{d}{2}$, $\gamma\neq\beta-\frac{d}{2}-1$, and $\gamma\neq 2\beta-\frac{d}{2}$,
then $\mathfrak{d}_{1,0}\neq\mathfrak{d}_{2,0}$, $\mathfrak{d}_{1,1}\neq\mathfrak{d}_{2,0}$, and $\mathfrak{d}_{1,0}\neq\mathfrak{d}_{2,1}$. Hence  $\mathcal{H}_{\sigma,\gamma}^{(q)}(z)$ has simple poles at $\mathfrak{d}_{1,0}=-\frac{d+2\gamma}{2\beta}$ and $\mathfrak{d}_{2,0}=-1$.
Therefore,
\begin{align}
 \mathbb{H}_{\sigma,\gamma}^{(q)}(r)&=-\sum_{j=1}^{2}\textnormal{Res}_{z=\mathfrak{d}_{j,0}}\left[\mathcal{H}_{\sigma,\gamma}^{(q)}(z)r^{-z}\right]-\textnormal{Res}_{z=\mathfrak{d}_{2,1}}\left[\mathcal{H}_{\sigma,\gamma}^{(q)}(z)r^{-z}\right]+O(r^{M})\nonumber \\
 & =-\left(\kappa_{1}(\mathfrak{d}_{1,0})^{q}r^{-\mathfrak{d}_{1,0}}+\kappa_{2}(\mathfrak{d}_{2,0})^{q}r^{-\mathfrak{d}_{2,0}}\right)\nonumber\\
 &\qquad\qquad\qquad\qquad-\textnormal{Res}_{z=-2}\left[\mathcal{H}_{\sigma,\gamma}^{(q)}(z)r^{-z}\right]+O(r^{\min\left(2,\frac{d+2\gamma+2}{2\beta}\right)})\nonumber \\
 & =-\left(-\frac{d+2\gamma}{2\beta}\right)^{q}\kappa_{1}r^{\frac{d+2\gamma}{2\beta}}-(-1)^{q}\cdot\kappa_{2}r\label{eq:second term-1}\\
 & \qquad\qquad\qquad\qquad-\textnormal{Res}_{z=-2}\left[\mathcal{H}_{\sigma,\gamma}^{(q)}(z)r^{-z}\right]+O(r^{\min\left(2,\frac{d+2\gamma+2}{2\beta}\right)})\nonumber
\end{align}
as $r\rightarrow0$. Note that $\kappa_{2}=0$ if $\gamma=\beta$ or $\sigma+\alpha\in\mathbb{N}$ so the second term of (\ref{eq:second term-1})
vanishes. In that case, observe that
$$
\textnormal{Res}_{z=-2}[\mathcal{H}_{\sigma,\gamma}^{(q)}(z)r^{-z}]=
\begin{cases}
O(r^{2}) &: \gamma\neq 2\beta-\frac{d}{2}-1\\
O(r^{2}|\ln r|) &: \gamma=2\beta-\frac{d}{2}-1.
\end{cases}
$$
Furthermore, if $\alpha=1$ and $\sigma=0$, then 
$$
\textnormal{Res}_{z=\mathfrak{d}_{2,k}}[\mathcal{H}_{\sigma,\gamma}^{(q)}(z)r^{-z}]=0
$$ 
for all $k\in\mathbb{Z}_{+}$. Thus (i) is proved.

(ii) Suppose $\gamma=\beta-\frac{d}{2}$. Then $\hat{\mathfrak{d}}_{0}=\mathfrak{d}_{1,0}=\mathfrak{d}_{2,0}$, $\kappa_{1}=\kappa_{2}$, and $\hat{\kappa}_{1}=\hat{\kappa}_{2}$. Note that
 $\mathcal{H}_{\sigma,\gamma}^{(q)}(z)$ has a pole at
$$
\hat{\mathfrak{d}}_{0}=-\frac{d+2\gamma}{2\beta}=-1
$$
of order 2. By \eqref{eq:cH^q},
\begin{align}
\textnormal{Res}_{z=-1}  \left[\mathcal{H}_{\sigma,\gamma}^{(q)}(z)r^{-z}\right]&=\lim_{z\rightarrow -1}\frac{d}{dz}\left((z+1)^{2}\mathcal{H}_{\sigma,\gamma}(z)z^{q}r^{-z}\right)\nonumber
\\ &=\left(-1\right)^{q}\cdot\hat{\kappa}_{2}r\ln r+\textnormal{Res}_{z=-1}\left[\mathcal{H}_{\sigma,\gamma}^{(q)}(z)\right]r\nonumber\\
&=\left(-1\right)^{q}\cdot\hat{\kappa}_{2}r\ln r+(-1)^{q}\cdot \left(\kappa_{2}-q\hat{\kappa}_{2}\right) r \nonumber \\
&=(-1)^{q}\cdot (\hat{\kappa}_{2}\ln r+\kappa_{2}) r - q(-1)^{q}\hat{\kappa}_{2}r\label{eq:log residue}
\end{align}
Therefore,
\begin{align*}
\mathbb{H}_{\sigma,\gamma}^{(q)}(r) =-(-1)^{q}\cdot (\hat{\kappa}_{2}\ln r+\kappa_{2}) r+q(-1)^{q}\hat{\kappa}_{2}r+
\begin{cases}
O(r^{\min(2,\frac{d+2\gamma+2}{2\beta})}) &: \beta\neq 1 \\
O(r^{2}|\ln r|) &:\beta =1 
\end{cases}
\end{align*}
as $r\rightarrow0$. If we additionally assume $\sigma+\alpha\in\mathbb{N}$, then $\hat{\kappa}_{2}=0$. Thus (ii) is proved.

(iii) Now we let $\gamma=\beta-\frac{d}{2}-1$. Then $\mathfrak{d}_{1,1}=\mathfrak{d}_{2,0}$ and
 $\mathcal{H}_{\sigma,\gamma}^{(q)}(z)$ has a pole at
$$
-\frac{d+2\gamma+2}{2\beta}=\mathfrak{d}_{1,1}=\mathfrak{d}_{2,0}=-1
$$
 of order 2, and  (\ref{eq:log residue}) holds. Also note that $\mathcal{H}_{\sigma,\gamma}^{(q)}(z)$
has a simple pole at 
$$
z=\mathfrak{d}_{1,0}=-\frac{d+2\gamma}{2\beta}=\frac{1}{\beta}-1.
$$
Therefore,
\begin{align*}
\mathbb{H}_{\sigma,\gamma}^{(q)}(r) & =-\textnormal{Res}_{z=-\frac{d+2\gamma}{2\beta}}\left[\mathcal{H}_{\sigma,\gamma}^{(q)}(z)r^{-z}\right]-\textnormal{Res}_{z=-1}\left[\mathcal{H}_{\sigma,\gamma}^{(q)}(z)r^{-z}\right]+O(r)\\
 & =-\left(-\frac{d+2\gamma}{2\beta}\right)^{q}\kappa_{1}r^{\frac{d+2\gamma}{2\beta}}+O(r|\ln r|)
\end{align*}
as $r\rightarrow0$. If $\sigma+\alpha\in\mathbb{N}$, then $\hat{\kappa}_{2}=0$ and
$$
\mathbb{H}_{\sigma,\gamma}^{(q)}(r) =-\left(-\frac{d+2\gamma}{2\beta}\right)^{q}\kappa_{1}r^{\frac{d+2\gamma}{2\beta}}+O(r)
$$
Thus (iii) is proved.

(iv) Finally, we assume $\gamma=2\beta-\frac{d}{2}$. Then $\mathfrak{d}_{1,0}=\mathfrak{d}_{2,1}$ and $\mathcal{H}_{\sigma,\gamma}^{(q)}(z)$ has a pole at
$$
-\frac{d+2\gamma}{2\beta}=\mathfrak{d}_{1,0}=\mathfrak{d}_{2,1}=-2
$$
of order 2. By \eqref{eq:cH^q},
\begin{align*}
&\textnormal{Res}_{z=-2}\left[\mathcal{H}_{\sigma,\gamma}^{(q)}(z)r^{-z}\right]\\
&=\lim_{z\rightarrow -2}\frac{d}{dz}\left((z+2)^{2}\mathcal{H}_{\sigma,\gamma}(z)z^{q}r^{-z}\right)\\
&=\left(-\frac{d+2\gamma}{2\beta}\right)\cdot \hat{\kappa}_{1}r^{2} \ln r+\lim_{z\rightarrow -2}\frac{d}{dz}\left((z+2)^{2}\mathcal{H}_{\sigma,\gamma}(z)z^{q}\right) r^{2}\\
&=\left(-\frac{d+2\gamma}{2\beta}\right)^{q}\cdot \left(\hat{\kappa}_{1}\ln r+\kappa_{1}\right)r^{2}+q\left(-\frac{d+2\gamma}{2\beta}\right)^{q-1}\hat{\kappa}_{1}r^2.
\end{align*}
Note that $\mathcal{H}_{\sigma,\gamma}^{(q)}(z)$ has a simple pole at $z=\mathfrak{d}_{2,0}=-1$. Therefore, 
\begin{align}
\mathbb{H}_{\sigma,\gamma}^{(q)}(r) & =-\textnormal{Res}_{z=-\frac{d+2\gamma}{2\beta}}\left[\mathcal{H}_{\sigma,\gamma}^{(q)}(z)r^{-z}\right]-\textnormal{Res}_{z=-1}\left[\mathcal{H}_{\sigma,\gamma}^{(q)}(z)r^{-z}\right]+O(r^{2})\nonumber\\
&= -\left(-\frac{d+2\gamma}{2\beta}\right)^{q}\cdot \left(\hat{\kappa}_{1}\ln r+\kappa_{1}\right)r^{2}-\kappa_{2}r+O(r^{2})\label{eq:5-8-1}\\
 & =-\left(-\frac{d+2\gamma}{2\beta}\right)^{q}\cdot \left(\hat{\kappa}_{1}\ln r+\kappa_{1}\right)r^{2}+O(r)\nonumber
\end{align}
as $r\rightarrow0$. Additionally, if $\sigma+\alpha\in\mathbb{N}$, then $\kappa_{2}=0$ in \eqref{eq:5-8-1} hence we obtain the desired result. The lemma is proved.
\end{proof}

\section{\label{sec:representation of solution} Explicit representation of
$p(t,x)$ and its derivatives}

Take the function $\mathbb{H}_{\sigma,\gamma}(r)$ from (\ref{eq:kernel representation}),
and define
$$
p_{\sigma,\gamma}(t,x):=\frac{2^{2\gamma}}{\pi^{d/2}}|x|^{-d-2\gamma}t^{-\sigma}\mathbb{H}_{\sigma,\gamma}\left(\frac{|x|^{2\beta}}{2^{2\beta}}t^{-\alpha}\right),\quad(t,x)\in(0,\infty)\times\mathbb{R}_{0}^{d}.
$$
Write
\begin{align}
					\label{eq:fundamental solution}
p(t,x):=p_{0,0}(t,x)=\frac{|x|^{-d}}{\pi^{d/2}}\mathbb{H}_{0,0}\left(\frac{|x|^{2\beta}}{2^{2\beta}}t^{-\alpha}\right).
\end{align}

Note that $\mathbb{H}_{\sigma,\gamma}(r)$ is a bounded function of $r\in (0,\infty)$. As a corollary of Lemmas \ref{lem:H large r} and \ref{lem:H small r},
we obtain the following upper estimates of $p_{\sigma,\gamma}(t,x)$ for $\gamma\in[0,\infty)$ and $\sigma\in\mathbb{R}$.

\begin{thm}
								\label{thm: derivative estimate}
Let $\alpha\in(0,2)$, $\beta\in(0,\infty)$,
$\gamma\in[0,\infty)$, and $\sigma\in\mathbb{R}$. Then for $|x|^{2\beta}t^{-\alpha}\geq1$
$$
|p_{\sigma,\gamma}(t,x)|\apprle|x|^{-d-2\gamma}t^{-\sigma}
$$
and for $|x|^{2\beta}t^{-\alpha}\leq1$
$$
\left|p_{\sigma,\gamma}(t,x)\right|\apprle\begin{cases}
t^{-\sigma-\frac{\alpha(d+2\gamma)}{2\beta}} & :\gamma<\beta-\frac{d}{2}\\
|x|^{-d-2\gamma+2\beta}t^{-\sigma-\alpha}\left(1+|\ln\left(|x|^{2\beta}t^{-\alpha}\right)|\right) & :\gamma=\beta-\frac{d}{2}\\
|x|^{-d-2\gamma+2\beta}t^{-\sigma-\alpha} & :\gamma>\beta-\frac{d}{2}.
\end{cases}
$$
 Furthermore,

(i) If $\sigma+\alpha\in\mathbb{N}$, then for $|x|^{2\beta}t^{-\alpha}\leq1$
$$
\left|p_{\sigma,\gamma}(t,x)\right|\apprle\begin{cases}
t^{-\sigma-\frac{\alpha(d+2\gamma)}{2\beta}} & :\gamma<2\beta-\frac{d}{2}\\
|x|^{-d-2\gamma+4\beta}t^{-\sigma-2\alpha}\left(1+\left|\ln\left(|x|^{2\beta}t^{-\alpha}\right)\right|\right) & :\gamma=2\beta-\frac{d}{2}\\
|x|^{-d-2\gamma+4\beta}t^{-\sigma-2\alpha} & :\gamma>2\beta-\frac{d}{2}.
\end{cases}
$$

(ii) If $\alpha=1$ and $\sigma=0$, then for $|x|^{2\beta}t^{-\alpha}\leq 1$

$$
|p_{0,\gamma}(t,x)|\apprle t^{-\frac{d+2\gamma}{2\beta}}.
$$

(iii) If $\gamma\in\mathbb{Z}_{+}$, then for $|x|^{2\beta}t^{-\alpha}\geq1$
$$
\left|p_{\sigma,\gamma}(t,x)\right|\apprle|x|^{-d-2\gamma-2\beta}t^{-\sigma+\alpha}.
$$

(iv) If $\beta\in{\mathbb{N}}$ and $\gamma=0$, then there exists a constant $c=c(\alpha,\beta,\sigma)>0$ such that for $|x|^{2\beta}t^{-\alpha}\geq 1$,
$$
|p_{\sigma,0}(t,x)|\apprle |x|^{-d}t^{-\sigma} \exp\left\{-c(t^{-\alpha}|x|^{2\beta})^{\frac{1}{2\beta-\alpha}}\right\}.
$$

(v) If $\gamma=\beta$, then for $|x|^{2\beta}t^{-\alpha}\geq1$
$$
\left|p_{\sigma,\beta}(t,x)\right|\apprle\begin{cases}
|x|^{-d-4\beta}t^{-\sigma+\alpha} & :\sigma\in\mathbb{N}\\
|x|^{-d-2\beta}t^{-\sigma} & :\sigma\in\mathbb{R}\setminus\mathbb{N}
\end{cases}
$$
and for $|x|^{2\beta}t^{-\alpha}\leq1$
$$
\left|p_{\sigma,\beta}(t,x)\right|\apprle\begin{cases}
t^{-\sigma-\alpha-\frac{\alpha d}{2\beta}} & :\frac{d}{2}<\beta\\
|x|^{-d+2\beta}t^{-\sigma-2\alpha}\left(1+\left|\ln\left(|x|^{2\beta}t^{-\alpha}\right)\right|\right) & :\frac{d}{2}=\beta\\
|x|^{-d+2\beta}t^{-\sigma-2\alpha} & :\frac{d}{2}>\beta.
\end{cases}
$$

\end{thm}

\begin{rem}
									\label{rem: integrable}
(i) By Theorem \ref{thm: derivative estimate},
$p_{\sigma,\gamma}(t,\cdot)\in L_{1}(\mathbb{R}^{d})$ for all $\gamma\in[0,\beta]$
and $\sigma\in\mathbb{R}$. Furthermore, if $\alpha=1$ and $\sigma=0$, then $p_{0,\gamma}(t,\cdot)\in L_{1}(\mathbb{R}^{d})$ for all $\gamma\in[0,\infty)$.

(ii) Observe that
\begin{align}
					\label{eq:scaling}
p_{\sigma,\gamma}(t,x)=t^{-\sigma-\frac{\alpha(d+2\gamma)}{2\beta}}p_{\sigma,\gamma}(1,t^{-\frac{\alpha}{2\beta}}x)
\end{align}
which implies
$$
\int_{0}^{T}\int_{\mathbb{R}^{d}}\left|p_{\sigma,\gamma}(t,x)\right|dxdt=\int_{0}^{T}t^{-\sigma-\frac{\alpha\gamma}{\beta}}\left(\int_{\mathbb{R}^{d}}\left|p_{\sigma,\gamma}(1,x)\right|dx\right)dt<\infty
$$
if $\sigma+\frac{\alpha\gamma}{\beta}<1$. Thus, under this condition, one
can consider Riemann-Liouville fractional integral and the Fourier-Laplace
transform of $p_{\sigma,\gamma}(t,x)$. However we do not use such transforms in this article.
\end{rem}

  The following
theorem handles the interchangeability of  $\Delta^{\gamma}$ and
$\mathbb{D}_{t}^{\sigma}$.

\begin{thm}
							\label{thm:time derivative}
Let $\gamma\in[0,\infty)$. For any $\sigma\in\mathbb{R}$,
$m\in\mathbb{N}$, and $\eta\in(-\infty,1)$,
$$
D_{t}^{m}p_{\sigma,\gamma}(t,x)=p_{\sigma+m,\gamma}(t,x),
$$
and
$$
\mathbb{D}_{t}^{\sigma}p_{\eta,\gamma}(t,x)=p_{\sigma+\eta,\gamma}(t,x).
$$

\end{thm}

\begin{proof}
First we show
\begin{align}
					\label{eq:derivative relation}
\frac{\partial}{\partial t}p_{\sigma,\gamma}(t,x)=p_{\sigma+1,\gamma}(t,x)
\end{align}
for any $\sigma\in\mathbb{R}$. By (\ref{eq:zGamma(z)=Gamma(z+1)}),
$$
\frac{\alpha z}{\Gamma(1-\sigma+\alpha z)}=\frac{1}{\Gamma(-\sigma+\alpha z)}+\frac{\sigma}{\Gamma(1-\sigma+\alpha z)}
$$
which implies
$$
\mathcal{H}_{\sigma,\gamma}(z)\alpha z=\mathcal{H}_{\sigma+1,\gamma}(z)+\sigma\mathcal{H}_{\sigma,\gamma}(z).
$$
Then by Proposition \ref{prop:derivatives of H(r)},
\begin{align*}
\frac{\partial}{\partial t}\mathbb{H}_{\sigma,\gamma}\left(\frac{|x|^{2\beta}}{2^{2\beta}}t^{-\alpha}\right) & =\frac{t^{-1}}{2\pi\textnormal{i}}\int_{L}\mathcal{H}_{\sigma,\gamma}(z)\alpha z\left(\frac{|x|^{2\beta}}{2^{2\beta}}t^{-\alpha}\right)^{-z}dz\\
 & =\frac{t^{-1}}{2\pi\textnormal{i}}\int_{L}\left(\mathcal{H}_{1+\sigma,\gamma}(z)+\sigma\mathcal{H}_{\sigma,\gamma}(z)\right)\left(\frac{|x|^{2\beta}}{2^{2\beta}}t^{-\alpha}\right)^{-z}dz\\
 & =t^{-1}\left[\mathbb{H}_{\sigma+1,\gamma}\left(\frac{|x|^{2\beta}}{2^{2\beta}}t^{-\alpha}\right)+\sigma\mathbb{H}_{\sigma,\gamma}\left(\frac{|x|^{2\beta}}{2^{2\beta}}t^{-\alpha}\right)\right].
\end{align*}
Therefore,
\begin{align*}
\frac{\partial}{\partial t}p_{\sigma,\gamma}(t,x) & =\frac{2^{2\gamma}}{\pi^{d/2}}|x|^{-d-2\gamma}t^{-\sigma}\left(\frac{\partial}{\partial t}\mathbb{H}_{\sigma,\gamma}\left(\frac{|x|^{2\beta}}{2^{2\beta}}t^{-\alpha}\right)-\sigma t^{-1}\mathbb{H}_{\sigma,\gamma}\left(\frac{|x|^{2\beta}}{2^{2\beta}}t^{-\alpha}\right)\right)\\
 & =\frac{2^{2\gamma}}{\pi^{d/2}}|x|^{-d-2\gamma}t^{-\sigma-1}\mathbb{H}_{\sigma+1,\gamma}\left(\frac{|x|^{2\beta}}{2^{2\beta}}t^{-\alpha}\right)\\
 & =p_{\sigma+1,\gamma}(t,x),
\end{align*}
and  (\ref{eq:derivative relation}) is proved. Thus to complete the proof of the theorem, due to \eqref{eqn 4.15} and (\ref{eq:derivative relation}),
it is sufficient to show that for $\sigma<0$ and $\eta<1$
$$
I_{t}^{|\sigma|}p_{\eta,\gamma}(t,x)=p_{\eta+\sigma,\gamma}(t,x).
$$
Take $\ell_{0}>-\frac{1-\eta}{\alpha}$ satisfying \eqref{eq:ell condition}.
For $\Re[z]>-\frac{1-\eta}{\alpha}$, it holds that
$$
\frac{1}{\Gamma(-\sigma)}\int_{0}^{t}(t-s)^{-\sigma-1}s^{-\eta+\alpha z}ds=\frac{\Gamma(1-\eta+\alpha z)}{\Gamma(1-\sigma-\eta+\alpha z)}t^{-\sigma-\eta+\alpha z}.
$$
Observe that
$$
\mathcal{H}_{\eta,\gamma}(z)\frac{\Gamma(1-\eta+\alpha z)}{\Gamma(1-\sigma-\eta+\alpha z)}=\mathcal{H}_{\eta+\sigma,\gamma}(z).
$$
By (\ref{eq:Stirling approx vertical}) and  the Fubini theorem,
\begin{align*}
I_{t}^{|\sigma|}p_{\eta,\gamma}(t,x) & =\int_{0}^{t}\frac{(t-s)^{-\sigma-1}}{\Gamma(-\sigma)}p_{\eta,\gamma}(s,x)ds\\
 & =\frac{2^{2\gamma}\pi^{-\frac{d}{2}}|x|^{-d-2\gamma}}{2\pi\textnormal{i}}\int_{0}^{t}\int_{L}\frac{(t-s)^{-\sigma-1}}{\Gamma(-\sigma)}\mathcal{H}_{\eta,\gamma}(z)s^{-\eta}\left(\frac{|x|^{2\beta}s^{-\alpha}}{2^{2\beta}}\right)^{-z}dzds\\
 & =\frac{2^{2\gamma}\pi^{-\frac{d}{2}}|x|^{-d-2\gamma}}{2\pi\textnormal{i}}\int_{L}\mathcal{H}_{\eta,\gamma}(z)\left(\frac{|x|^{2\beta}}{2^{2\beta}}\right)^{-z}\left[\int_{0}^{t}\frac{(t-s)^{-\sigma-1}}{\Gamma(-\sigma)}s^{-\eta+\alpha z}ds\right]dz\\
 & =\frac{2^{2\gamma}\pi^{-\frac{d}{2}}|x|^{-d-2\gamma}}{2\pi\textnormal{i}}t^{-\sigma}\int_{L}\mathcal{H}_{\eta+\sigma,\gamma}(z)\left(\frac{|x|^{2\beta}t^{-\alpha}}{2^{2\beta}}\right)^{-z}dz\\
 & =2^{2\gamma}\pi^{-\frac{d}{2}}|x|^{-d-2\gamma}t^{-\eta-\sigma}\mathbb{H}_{\eta+\sigma,\gamma}\left(\frac{|x|^{2\beta}}{2^{2\beta}}t^{-\alpha}\right)\\
 & =p_{\eta+\sigma,\gamma}(t,x).
\end{align*}
The theorem is proved.
\end{proof}

\begin{rem}
							\label{rem:diff Fourier}
Let $0<\varepsilon<T$. Then by Theorems \ref{thm: derivative estimate} and \ref{thm:time derivative}, $p(t,x)$ and $\frac{\partial p}{\partial t}(t,x)$
are integrable in $x\in\mathbb{R}^{d}$ uniformly in $t\in[\varepsilon,T]$.
Thus
$$
\frac{\partial}{\partial t}\int_{\mathbb{R}^{d}}e^{-\textnormal{i}(x,\xi)}p(t,x)dx=\int_{\mathbb{R}^{d}}e^{-\textnormal{i}(x,\xi)}\frac{\partial p}{\partial t}(t,x)dx.
$$

\end{rem}

To estimate the classical derivatives of $p(t,x)$, we need the following
theorem.
\begin{thm}
								\label{thm: usual derivative estimate 2}
Let $\alpha,\beta,\gamma,\sigma$ be given as in Theorem \ref{thm: derivative estimate} and $n\in\mathbb{N}$. Then for $|x|^{2\beta}t^{-\alpha}\geq1$
$$
\left|D_{x}^{n}p_{\sigma,\gamma}(t,x)\right|\apprle\begin{cases}
|x|^{-d-n}t^{-\sigma}\exp\left\{-(|x|^{2\beta}t^{-\alpha})^{\frac{1}{2\beta-\alpha}}\right\}& : \beta\in\mathbb{N},\gamma=0\\
|x|^{-d-2\gamma-2\beta-n}t^{-\sigma+\alpha} & :\gamma\in\mathbb{Z}_{+}\ \mbox{or}\ \sigma\in\mathbb{N}\\
|x|^{-d-2\gamma-n}t^{-\sigma} & :\mbox{otherwise,}
\end{cases}
$$
and for $|x|^{2\beta}t^{-\alpha}\leq1$
$$
|D_{x}^{n}p_{\sigma,\gamma}(t,x)|\apprle\begin{cases}
|x|^{2-n}t^{-\sigma-\frac{\alpha(d+2\gamma+2)}{2\beta}} & :\gamma<\beta-\frac{d}{2}-1\\
|x|^{2-n}t^{-\sigma-\alpha}(1+|\ln|x|^{2\beta}t^{-\alpha}|) & :\gamma=\beta-\frac{d}{2}-1\\
|x|^{-d-2\gamma+2\beta-n}t^{-\sigma-\alpha} & :\gamma>\beta-\frac{d}{2}-1.
\end{cases}
$$
Additionally, for $|x|^{2\beta}t^{-\alpha}\leq1$, the followings hold:

(i) If $\alpha=1$ and $\sigma=0$,
$$
\left|D_{x}^{n}p_{0,\gamma}(t,x)\right| \apprle |x|^{2-n} t^{-\frac{d+2\gamma+2}{2\beta}}.
$$

(ii) If $\sigma+\alpha\in\mathbb{N}$,
$$
\left|D_{x}^{n}p_{\sigma,\gamma}(t,x)\right|\apprle\begin{cases}
|x|^{2-n}t^{-\sigma-\frac{\alpha(d+2\gamma+2)}{2\beta}}&:\gamma<2\beta-\frac{d}{2}-1\\
|x|^{2-n}t^{-\sigma-2\alpha}(1+\left|\ln|x|^{2\beta}t^{-\alpha}\right|) & :\gamma=2\beta-\frac{d}{2}-1\\
|x|^{-d-2\gamma+4\beta-n}t^{-\sigma-2\alpha} & :\gamma>2\beta-\frac{d}{2}-1.
\end{cases}
$$

(iii) If $\gamma=\beta$,
$$
\left|D_{x}^{n}p_{\sigma,\beta}(t,x)\right|\apprle\begin{cases}
|x|^{2-n}t^{-\sigma-\alpha-\frac{\alpha(d+2)}{2\beta}}& :\frac{d}{2}+1<\beta\\
|x|^{2-n}t^{-\sigma-2\alpha}(1+\left|\ln|x|^{2\beta}t^{-\alpha}\right|) & :\frac{d}{2}+1=\beta\\
|x|^{-d+2\beta-n}t^{-\sigma-2\alpha} & :\frac{d}{2}+1>\beta.
\end{cases}
$$

\end{thm}

\begin{proof}
Write $R=R(t,x):=2^{-2\beta}|x|^{2\beta}t^{-\alpha}$ and recall $\mathbb{H}_{\sigma,\gamma}^{(q)}$ from (\ref{eq:H^q}).
By Proposition \ref{prop:derivatives of H(r)} and (\ref{eq:zGamma(z)=Gamma(z+1)}),
\begin{align*}
D_{x^{i}} & \mathbb{H}_{\sigma,\gamma}^{(q)}\left(R\right)=-2\beta\frac{x^{i}}{|x|^{2}}\mathbb{H}_{\sigma,\gamma}^{(q+1)}\left(R\right)
\end{align*}
for each $q\in\mathbb{Z}_{+}$ and $i=1,\ldots,d$. Hence
\begin{align*}
|D_{x^{i}} & p_{\sigma,\gamma}(t,x)|\\
 & =\Bigg|-\frac{2^{2\gamma}}{\pi^{d/2}}x^{i}|x|^{-d-2\gamma-2}t^{-\sigma}\Bigg((d+2\gamma)\mathbb{H}_{\sigma,\gamma}\left(R\right)+2\beta\mathbb{H}_{\sigma,\gamma}^{(1)}\left(R\right)\Bigg)\Bigg|\\
 & \leq C|x|^{-d-2\gamma-1}t^{-\sigma}\left|(d+2\gamma)\mathbb{H}_{\sigma,\gamma}\left(R\right)+2\beta\mathbb{H}_{\sigma,\gamma}^{(1)}\left(R\right)\right|
\end{align*}
and
\begin{align*}
 & |D_{x^{j}}D_{x^{i}}p_{\sigma,\gamma}(t,x)|\\
 & =\left|D_{x^{j}}\left\{ -\frac{2^{2\gamma}}{\pi^{d/2}}x^{i}|x|^{-d-2\gamma-2}t^{-\sigma}\Bigg((d+2\gamma)\mathbb{H}_{\sigma,\gamma}\left(R\right)+2\beta\mathbb{H}_{\sigma,\gamma}^{(1)}\left(R\right)\right) \right|\\
 & =\frac{2^{2\gamma}}{\pi^{d/2}}t^{-\sigma}\Bigg|-\delta_{ij}|x|^{-d-2\gamma-2}\left\{ (d+2\gamma)\mathbb{H}_{\sigma,\gamma}(R)+2\beta\mathbb{H}_{\sigma,\gamma}^{(1)}(R)\right\} \\
 & \qquad\qquad\qquad\qquad+(d+2\gamma+2)x^{i}x^{j}|x|^{-d-2\gamma-4}\left\{ (d+2\gamma)\mathbb{H}_{\sigma,\gamma}(R)+2\beta\mathbb{H}_{\sigma,\gamma}^{(1)}(R)\right\} \\
 & \qquad\qquad\qquad\qquad\qquad+2\beta x^{i}x^{j}|x|^{-d-2\gamma-4}\left\{ (d+2\gamma)\mathbb{H}_{\sigma,\gamma}^{(1)}(R)+2\beta\mathbb{H}_{\sigma,\gamma}^{(2)}(R)\right\} \Bigg|\\
 & \leq C|x|^{-d-2\gamma-2}t^{-\sigma}\sum_{q=1}^{2}\left|(d+2\gamma)\mathbb{H}_{\sigma,\gamma}^{(q-1)}(R)+2\beta\mathbb{H}_{\sigma,\gamma}^{(q)}(R)\right|.
\end{align*}
Inductively, for any $n\in\mathbb{N}$,
\begin{align}
					\label{eq:lemma-derivative 1}
\left|D_{x}^{n}p_{\sigma,\gamma}(t,x)\right|\leq C|x|^{-d-2\gamma-n}t^{-\sigma}\sum_{q=1}^{n}\left|(d+2\gamma)\mathbb{H}_{\sigma,\gamma}^{(q-1)}(R)+2\beta\mathbb{H}_{\sigma,\gamma}^{(q)}(R)\right|.
\end{align}

By Lemma \ref{lem:H large r}, \ref{lem:H^q large}, and \eqref{eq:lemma-derivative 1},
$$
\left|D_{x}^{n}p_{\sigma,\gamma}(t,x)\right|\apprle\begin{cases}
t^{\frac{\alpha (d+n)}{2\beta}-\sigma}\exp\left\{-c(|x|^{2\beta}t^{-\alpha})^{\frac{1}{2\beta-\alpha}}\right\} & : \gamma=0, \beta\in\mathbb{N}\\
|x|^{-d-2\gamma-2\beta-n}t^{-\sigma+\alpha} & :\gamma\in\mathbb{Z}_{+}\ \mbox{or}\ \sigma\in\mathbb{N}\\
|x|^{-d-2\gamma-n}t^{-\sigma} & :\mbox{otherwise}
\end{cases}
$$
for $|x|^{2\beta}t^{-\alpha}\geq1$. To estimate the upper bounds for $|x|^{2\beta}t^{-\alpha}\leq1$, observe
that
$$
(d+2\gamma)\left(-\frac{d+2\gamma}{2\beta}\right)^{q-1}+2\beta\left(-\frac{d+2\gamma}{2\beta}\right)^{q}=0.
$$
Thus, by Lemma \ref{lem:H^q small} and (\ref{eq:lemma-derivative 1}),
$$
|D_{x}^{n}p_{\sigma,\gamma}(t,x)|\apprle\begin{cases}
|x|^{2-|n|}t^{-\sigma-\frac{\alpha(d+2\gamma+2)}{2\beta}} & :\gamma<\beta-\frac{d}{2}-1\\
|x|^{2-|n|}t^{-\sigma-\alpha}(1+|\ln|x|^{2\beta}t^{-\alpha}|) & :\gamma=\beta-\frac{d}{2}-1\\
|x|^{-d-2\gamma+2\beta-|n|}t^{-\sigma-\alpha} & :\gamma>\beta-\frac{d}{2}-1
\end{cases}
$$
for $R\leq1$.

Additionally, if $\sigma+\alpha\in\mathbb{N}$, by Lemma \ref{lem:H^q small},
$$
\left|D_{x}^{n}p_{\sigma,\gamma}(t,x)\right|\apprle\begin{cases}
|x|^{2-n}t^{-\sigma-\frac{\alpha(d+2\gamma+2)}{2\beta}}&:\gamma<2\beta-\frac{d}{2}-1\\
|x|^{2-|n|}t^{-\sigma-\alpha}(1+\left|\ln|x|^{2\beta}t^{-\alpha}\right|) & :\gamma=2\beta-\frac{d}{2}-1\\
|x|^{-d-2\gamma+4\beta-|n|}t^{-\sigma-2\alpha} & :\gamma>2\beta-\frac{d}{2}-1
\end{cases}
$$
for $R\leq1$. Also, if $\gamma=\beta$,
$$
\left|D_{x}^{n}p_{\sigma,\beta}(t,x)\right|\apprle\begin{cases}
|x|^{2-n}t^{-\sigma-\alpha-\frac{\alpha(d+2)}{2\beta}} &: \frac{d}{2}+1<\beta \\
|x|^{2-|n|}t^{-\sigma-2\alpha}(1+\left|\ln|x|^{2\beta}t^{-\alpha}\right|) & :\frac{d}{2}+1=\beta\\
|x|^{-d+2\beta-|n|}t^{-\sigma-2\alpha} & :\frac{d}{2}+1>\beta
\end{cases}
$$
for $R\leq1$. The theorem is proved.
\end{proof}

\section{\label{sec: proof of main theorem}Proofs of Theorems \ref{thm:main result 1},
\ref{thm:derivative of kernel} and  \ref{cor:main result 1}}

By  Remark \ref{rem: integrable},   $p_{\sigma,\gamma}(t,\cdot)\in L_{1}(\mathbb{R}^{d})$ if $\gamma\in[0,\beta]$
and $\sigma\in\mathbb{R}$, and $p_{0,\gamma}(t,\cdot)\in L_{1}(\mathbb{R}^{d})$ if $\alpha=1$, $\gamma\in [0,\infty)$, and $\sigma=0$. Due to Lemmas \ref{lem:H large r}, \ref{lem:H small r}, Theorems \ref{thm: derivative estimate}, \ref{thm:time derivative}, and \ref{thm: usual derivative estimate 2}, to prove our desired results, it is enough to show
\begin{align}
					\label{eq:goal}
\mathcal{F}\left\{ p_{\sigma,\gamma}(t,\cdot)\right\} =|\xi|^{2\gamma}t^{-\sigma}E_{\alpha,1-\sigma}(-t^{\alpha}|\xi|^{2\beta}),
\end{align}
which is equivalent to
$$
p_{\sigma,\gamma}(t,x)=\Delta^{\gamma}p_{\sigma,0}(t,x)=\Delta^{\gamma}\mathbb{D}_{t}^{\sigma}p(t,x).
$$

We devide the proof into the cases $d\geq2$ and $d=1$.

\subsection*{Case 1: $d\geq2$}

We choose $\ell_{0}\in\mathbb{R}$ in (\ref{eq:kernel representation})
such that
\begin{align}
				\label{ell 0}
\max(-2,-1-\frac{d-1}{4\beta})<\ell_{0}<-1
\end{align}
if $\gamma=\beta$,
\begin{align}
				\label{ell 1}
\max(-1,-\frac{\gamma}{\beta}-\frac{d-1}{4\beta})<\ell_{0}<-\frac{\gamma}{\beta}
\end{align}
if $\gamma\in(0,\beta)$,
\begin{align}
				\label{ell 2}
\max(-1,-\frac{1}{\alpha},-\frac{d-1}{4\beta})<\ell_{0}<0
\end{align}
if $\gamma=0$. If $\alpha=1$, $\gamma\neq 0$, and $\sigma=0$, then we take $\ell_{0}$ such that
\begin{align}
				\label{ell 3}
-\frac{\gamma}{\beta}-\frac{d-1}{4\beta} < \ell_{0} < -\frac{\gamma}{\beta}.
\end{align}
Under the above restriction on $\ell_{0}$, (\ref{eq:kernel representation})
is well-defined and the value of $p_{\sigma,\gamma}(t,x)$ is independent
of the choice of $\ell_{0}$.

Due to the Fourier transform for radial
function (see \cite[Theorem IV.3.3]{stein1971introduction})
\begin{align*}
 & \mathcal{F}\left\{ p_{\sigma,\gamma}(t,\cdot)\right\} (\xi)\\
 & =\frac{2^{\frac{d}{2}+2\gamma}}{|\xi|^{\frac{d}{2}-1}}t^{-\sigma}\int_{0}^{\infty}\rho^{-\frac{d}{2}-2\gamma}\mathbb{H}_{\sigma,\gamma}\left(\frac{\rho^{2\beta}}{2^{2\beta}}t^{-\alpha}\right)J_{\frac{d}{2}-1}(|\xi|\rho)d\rho
\end{align*}
where $J_{\frac{d}{2}-1}$ is the Bessel function of the first kind
of order $\frac{d}{2}-1$, (i.e.) for $r\in[0,\infty)$,
$$
J_{\frac{d}{2}-1}(r)=\sum_{k=0}^{\infty}\frac{(-1)^{k}}{k!\Gamma(k+\frac{d}{2})}\left(\frac{r}{2}\right)^{2k-1+\frac{d}{2}}.
$$
It is well-known if $m>-1$ then
$$
J_{m}(t)=\begin{cases}
O(t^{m}) & :t\rightarrow0+\\
O(t^{-1/2}) & :t\rightarrow\infty.
\end{cases}
$$
 Due to \eqref{ell 0}-\eqref{ell 3},
\begin{align*}
\int_{0}^{\infty}\left|\rho^{-\frac{d}{2}-2\gamma-2\beta\ell_{0}}J_{\frac{d}{2}-1}(|\xi|\rho)\right|d\rho & \leq\int_{0}^{1}\rho^{-2\beta\ell_{0}-2\gamma-1}d\rho+\int_{1}^{\infty}\rho^{-\frac{d}{2}-2\gamma-2\beta\ell_{0}-\frac{1}{2}}d\rho<\infty.
\end{align*}
Recall $\Re[z]=\ell_{0}$ along $L$. By the above and (\ref{eq:Stirling approx vertical}),
$$
\int_{0}^{\infty}\int_{L}\left|\rho^{-\frac{d}{2}}\mathcal{H}_{\sigma,\gamma}(z)\left(\frac{\rho^{2\beta}t^{-\alpha}}{2^{2\beta}}\right)^{-z}J_{\frac{d}{2}-1}(|\xi|\rho)\right||dz|d\rho<\infty.
$$
Therefore, by the Fubini theorem
\begin{align*}
 & \int_{0}^{\infty}\rho^{-\frac{d}{2}-2\gamma}\mathbb{H}_{\sigma,\gamma}\left(\frac{\rho^{2\beta}}{2^{2\beta}}t^{-\alpha}\right)J_{\frac{d}{2}-1}(|\xi|\rho)d\rho\\
 & =\frac{1}{2\pi\textnormal{i}}\int_{0}^{\infty}\rho^{-\frac{d}{2}-2\gamma}\left[\int_{L}\mathcal{H}_{\sigma,\gamma}(z)\left(\frac{\rho^{2\beta}t^{-\alpha}}{2^{2\beta}}\right)^{-z}dz\right]J_{\frac{d}{2}-1}(|\xi|\rho)d\rho\\
 & =\frac{1}{2\pi\textnormal{i}}\int_{L}\left[\int_{0}^{\infty}\rho^{-\frac{d}{2}-2\gamma-2\beta z}J_{\frac{d}{2}-1}(|\xi|\rho)d\rho\right]\mathcal{H}_{\sigma,\gamma}(z)\left(\frac{t^{-\alpha}}{2^{2\beta}}\right)^{-z}dz.
\end{align*}
By using the formula \cite[(11.4.16)]{Bessel},
$$
\int_{0}^{\infty}\rho^{-\frac{d}{2}-2\gamma-2\beta z}J_{\frac{d}{2}-1}(|\xi|\rho)d\rho=2^{-\frac{d}{2}-2\gamma-2\beta z}|\xi|^{\frac{d}{2}+2\gamma+2\beta z-1}\frac{\Gamma(-\gamma-\beta z)}{\Gamma(\frac{d}{2}+\gamma+\beta z)}
$$
we have
\begin{align*}
\mathcal{H}_{\sigma,\gamma}(z)\frac{\Gamma(-\gamma-\beta z)}{\Gamma(\frac{d}{2}+\gamma+\beta z)} & =\frac{\Gamma(z+1)\Gamma(-z)}{\Gamma(1-\sigma+\alpha z)}.
\end{align*}
Hence
\begin{align*}
 & \frac{2^{\frac{d}{2}+2\gamma}}{|\xi|^{\frac{d}{2}-1}}t^{-\sigma}\cdot\frac{1}{2\pi\textnormal{i}}\int_{L}\left[\int_{0}^{\infty}\rho^{-\frac{d}{2}-2\gamma-2\beta z}J_{\frac{d}{2}-1}(|\xi|\rho)d\rho\right]\mathcal{H}_{\sigma,\gamma}(z)\left(\frac{t^{-\alpha}}{2^{2\beta}}\right)^{-z}dz\\
 & =\frac{2^{\frac{d}{2}+2\gamma}}{|\xi|^{\frac{d}{2}-1}}t^{-\sigma}\cdot\frac{2^{-\frac{d}{2}-2\gamma}}{2\pi\textnormal{i}}|\xi|^{\frac{d}{2}+2\gamma-1}\int_{L}\frac{\Gamma(z+1)\Gamma(-z)}{\Gamma(1-\sigma+\alpha z)}\left(|\xi|^{-2\beta}t^{-\alpha}\right)^{-z}dz\\
 & =|\xi|^{2\gamma}t^{-\sigma}\cdot\frac{1}{2\pi\textnormal{i}}\int_{-L}\frac{\Gamma(1-z)\Gamma(z)}{\Gamma(1-\sigma-\alpha z)}\left(|\xi|^{2\beta}t^{\alpha}\right)^{-z}dz\\
 & =|\xi|^{2\gamma}t^{-\sigma}\textnormal{H}_{12}^{11}\left[|\xi|^{2\beta}t^{\alpha}\ \Big|\begin{array}{cc}
(0,1)\\
(0,1) & (\sigma,\alpha)
\end{array}\right]\\
 & =|\xi|^{2\gamma}t^{-\sigma}E_{\alpha,1-\sigma}(-|\xi|^{2\beta}t^{\alpha}).
\end{align*}
Therefore,
$$
\mathcal{F}\left\{ p_{\sigma,\gamma}(t,\cdot)\right\} =|\xi|^{2\gamma}t^{-\sigma}E_{\alpha,1-\sigma}(-|\xi|^{2\beta}t^{\alpha}).
$$

\subsection*{Case 2: $d=1,\gamma\in(0,\beta)$.}

We choose $\ell_{0}$ such that
$$
\max(-1,-\frac{\gamma}{\beta}-\frac{1}{2\beta})<\ell_{0}<0.
$$
Since $p_{\sigma,\gamma}(t,x)$ is an even function,
\begin{align*}
\mathcal{F}\left\{ p_{\sigma,\gamma}(t,\cdot)\right\}  & =\int_{-\infty}^{\infty}e^{-\textnormal{i}\xi x}p_{\sigma,\gamma}(t,x)dx\\
 & =2\int_{0}^{\infty}p_{\sigma,\gamma}(t,x)\cos(\xi x)dx\\
 & =2\left(\int_{0}^{t^{\alpha/2\beta}}p_{\sigma,\gamma}(t,x)\cos(\xi x)dx+\int_{t^{\alpha/2\beta}}^{\infty}p_{\sigma,\gamma}(t,x)\cos(\xi x)dx\right)\\
 & =2t^{-\sigma-\frac{\alpha\gamma}{\beta}}\left(\int_{0}^{1}p_{\sigma,\gamma}(1,x)\cos(t^{\frac{\alpha}{2\beta}}\xi x)dx+\int_{1}^{\infty}p_{\sigma,\gamma}(1,x)\cos(t^{\frac{\alpha}{2\beta}}\xi x)dx\right).
\end{align*}
The last equality holds due to (\ref{eq:scaling}). Set Bromwich contours
$$
L_{<c}:=\left\{ z\in\mathbb{C}:\Re[z]=\ell_{<c}\right\} ,\quad L_{>c}:=\left\{ z\in\mathbb{C}:\Re[z]=\ell_{>c}\right\}
$$
where
$$
\max(-1,-\frac{\gamma}{\beta}-\frac{1}{2\beta})<\ell_{<c}<-\frac{\gamma}{\beta},\quad -\frac{\gamma}{\beta}<\ell_{>c}<0.
$$
By (\ref{eq:Stirling approx vertical}),
\begin{align*}
\int_{0}^{1}\int_{L_{<c}} & \left|2^{2\beta z}\mathcal{H}_{\sigma,\gamma}(z)x^{-2\gamma-2\beta z-1}\right||dz|dx\\
 & +\int_{1}^{\infty}\int_{L_{>c}}\left|2^{2\beta z}\mathcal{H}_{\sigma,\gamma}(z)x^{-2\gamma-2\beta z-1}\right||dz|dx<\infty.
\end{align*}
Therefore, by the Fubini theorem,

\begin{align*}
\int_{0}^{1} & p_{\sigma,\gamma}(1,x)\cos(t^{\frac{\alpha}{2\beta}}\xi x)dx\\
 & =\frac{2^{2\gamma}\pi^{-1/2}}{2\pi\textnormal{i}}\int_{0}^{1}\int_{L_{<c}}2^{2\beta z}\mathcal{H}_{\sigma,\gamma}(z)x^{-2\gamma-2\beta z-1}\cos(t^{\frac{\alpha}{2\beta}}\xi x)dzdx\\
 & =\frac{2^{2\gamma}\pi^{-1/2}}{2\pi\textnormal{i}}\int_{L_{<c}}2^{2\beta z}\mathcal{H}_{\sigma,\gamma}(z)\left[\int_{0}^{1}x^{-2\gamma-2\beta z-1}\cos(t^{\frac{\alpha}{2\beta}}\xi x)dx\right]dz.
\end{align*}
Similarly,
\begin{align*}
\int_{1}^{\infty} & p_{\sigma,\gamma}(1,x)\cos(t^{\frac{\alpha}{2\beta}}\xi x)dx\\
 & =\frac{2^{2\gamma}\pi^{-1/2}}{2\pi\textnormal{i}}\int_{L_{>c}}2^{2\beta z}\mathcal{H}_{\sigma,\gamma}(z)\left[\int_{1}^{\infty}x^{-2\gamma-2\beta z-1}\cos(t^{\frac{\alpha}{2\beta}}\xi x)dx\right]dz.
\end{align*}

For $\eta\in\mathbb{R}$, it holds that (see \cite[(1.8.1.1), (1.8.1.2)]{PAY})
$$
\int_{0}^{1}x^{-2\lambda-1}\cos(\eta x)dx=-\frac{_{1}F_{2}\left(-\lambda;\frac{1}{2},1-\lambda;-\frac{\eta^{2}}{4}\right)}{2\lambda},\quad\Re[\lambda]<0
$$
and
\begin{align*}
\int_{1}^{\infty} & x^{-2\lambda-1}\cos(\eta x)dx=\frac{\Gamma(-2\lambda)\cos\lambda\pi}{|\eta|^{-2\lambda}}+\frac{_{1}F_{2}\left(-\lambda;\frac{1}{2},1-\lambda;-\frac{\eta^{2}}{4}\right)}{2\lambda},\quad\Re[\lambda]>-\frac{1}{2},
\end{align*}
where $_{1}F_{2}\left(-\lambda;\frac{1}{2},1-\lambda;-\frac{\eta^{2}}{4}\right)$
denotes the general hypergeometric function, i.e.
$$
_{1}F_{2}(p;q,r;z)=\frac{\Gamma(q)\Gamma(r)}{\Gamma(p)}\sum_{k=0}^{\infty}\frac{\Gamma(p+k)}{\Gamma(q+k)\Gamma(r+k)}\frac{z^{k}}{k!},\quad p\in\mathbb{C},\ q,r\in\mathbb{C}\setminus\{0,-1,-2,\cdots\}.
$$
Observe that for $\gamma+\beta z\in\mathbb{C}\setminus\mathbb{N}$,
\begin{align*}
_{1}F_{2} & \left(-\gamma-\beta z;\frac{1}{2},1-\gamma-\beta z;-\frac{t^{\frac{\alpha}{\beta}}|\xi|^{2}}{4}\right)\\
 & =\frac{\sqrt{\pi}\Gamma(1-\gamma-\beta z)}{\Gamma(-\gamma-\beta z)}\sum_{k=0}^{\infty}\frac{\Gamma(-\gamma-\beta z+k)}{\Gamma(\frac{1}{2}+k)\Gamma(1-\gamma-\beta z+k)}\frac{\left(-t^{\frac{\alpha}{\beta}}|\xi|^{2}/4\right)^{k}}{k!}\\
 & =-\sum_{k=0}^{\infty}\frac{\sqrt{\pi}(\gamma+\beta z)}{\Gamma(\frac{1}{2}+k)k!(-\gamma-\beta z+k)}\left(-\frac{t^{\frac{\alpha}{\beta}}|\xi|^{2}}{4}\right)^{k}\\
 & =1-\sum_{k=1}^{\infty}\frac{\gamma+\beta z}{(2k-1)!2k(-\gamma-\beta z+k)}\left(-\frac{t^{\frac{\alpha}{\beta}}|\xi|^{2}}{4}\right)^{k}=:F_{t,\xi}(z).
\end{align*}
Since the series in $F_{t,\xi}(z)$ converges absolutely for fixed
$t,\xi$, one can easily see that $F_{t,\xi}(z)$ is holomorphic in
$\{z\in\mathbb{C}:-\frac{2\gamma-1}{2\beta}<\Re[z]<-\frac{2\gamma+1}{2\beta}\}$
by Morera's theorem. So we obtain
$$
\int_{0}^{1}p_{\sigma,\gamma}(1,x)\cos(t^{\frac{\alpha}{2\beta}}\xi x)dx=-\frac{2^{2\gamma}\pi^{-1/2}}{2\pi\textnormal{i}}\int_{L_{<c}}2^{2\beta z}\frac{\mathcal{H}_{\sigma,\gamma}(z)}{2(\gamma+\beta z)}F_{t,\xi}(z)dz,
$$
and
\begin{align*}
 & \int_{1}^{\infty}p_{\sigma,\gamma}(1,x)\cos(t^{\frac{\alpha}{2\beta}}\xi x)dx\\
 & =\frac{2^{2\gamma}\pi^{-1/2}}{2\pi\textnormal{i}}\Bigg(\int_{L_{>c}}2^{2\beta z}\mathcal{H}_{\sigma,\gamma}(z)\Gamma(-2\gamma-2\beta z)\cos\left((\gamma+\beta z)\pi\right)\left|t^{\frac{\alpha}{2\beta}}\xi\right|^{2\gamma+2\beta z}dz\\
 & \qquad\qquad\qquad\qquad\qquad+\int_{L_{>c}}2^{2\beta z}\frac{\mathcal{H}_{\sigma,\gamma}(z)}{2(\gamma+\beta z)}F_{t,\xi}(z)dz\Bigg).
\end{align*}
If $\gamma\neq\beta$, by (\ref{eq:zGamma(z)=Gamma(z+1)}),
$$
\frac{\mathcal{H}_{\sigma,\gamma}(z)}{2(\gamma+\beta z)}=-\frac{\Gamma(\frac{1}{2}+\gamma+\beta z)\Gamma(1+z)\Gamma(-z)}{2\Gamma(1-\gamma-\beta z)\Gamma(1-\sigma+\alpha z)}.
$$
Thus $z=-\frac{\gamma}{\beta}$ is a removable singularity. These
facts lead to
\begin{align}
					\label{eq:residue vanishes}
\textnormal{Res}_{z=-\frac{\gamma}{\beta}}\left[2^{2\beta z}\frac{\mathcal{H}_{\sigma,\gamma}(z)}{2(\gamma+\beta z)}F_{t,\xi}(z)\right]=0
\end{align}
and
$$
\frac{1}{2\pi\textnormal{i}}\int_{L_{>c}}2^{2\beta z}\frac{\mathcal{H}_{\sigma,\gamma}(z)}{2(\gamma+\beta z)}F_{t,\xi}(z)dz=\frac{1}{2\pi\textnormal{i}}\int_{L_{<c}}2^{2\beta z}\frac{\mathcal{H}_{\sigma,\gamma}(z)}{2(\gamma+\beta z)}F_{t,\xi}(z)dz.
$$
Therefore,
\begin{align*}
 & \int_{0}^{\infty}p_{\sigma,\gamma}(1,x)\cos(t^{\frac{\alpha}{2\beta}}\xi x)dx\\
 & =\frac{2^{2\gamma}\pi^{-1/2}t^{\frac{\alpha\gamma}{\beta}}|\xi|^{2\gamma}}{2\pi\textnormal{i}}\int_{L_{>c}}2^{2\beta z}\mathcal{H}_{\sigma,\gamma}(z)\Gamma(-2\gamma-2\beta z)\cos\left((\gamma+\beta z)\pi\right)\left(t^{\alpha}|\xi|^{2\beta}\right)^{z}dz.
\end{align*}
By (\ref{eq:gamma property}),
$$
\mathcal{H}_{\sigma,\gamma}(z)\Gamma(-2\gamma-2\beta z)\cos\left((\gamma+\beta z)\pi\right)=\pi^{1/2}2^{-2\gamma-2\beta z-1}\frac{\Gamma(1+z)\Gamma(-z)}{\Gamma(1-\sigma+\alpha z)}.
$$
Hence
\begin{align*}
\mathcal{F}\left\{ p_{\sigma,\gamma}(t,\cdot)\right\} (\xi) & =2t^{-\sigma-\frac{\alpha\gamma}{\beta}}\int_{0}^{\infty}p_{\sigma,\gamma}(1,x)\cos(t^{\frac{\alpha}{2\beta}}\xi x)dx\\
 & =\frac{|\xi|^{2\gamma}t^{-\sigma}}{2\pi\textnormal{i}}\int_{L}\mathcal{H}_{\sigma,\gamma}(z)\frac{\Gamma(1+z)\Gamma(-z)}{\Gamma(1-\sigma+\alpha z)}\left(t^{\alpha}|\xi|^{2\beta}\right)^{z}dz\\
 & =|\xi|^{2\gamma}t^{-\sigma}E_{\alpha,1-\sigma}\left(-|\xi|^{2\beta}t^{\alpha}\right).
\end{align*}

\subsection*{Case 3: $d=1,\gamma=0$.}

Let
$$
\max(-1,-\frac{1}{\alpha},-\frac{1}{2\beta})<\ell_{0}<1.
$$
Again, we follow the argument in Case 2. Note that
\begin{align*}
\frac{\mathcal{H}_{\sigma,0}(z)}{2\beta z} & =-\frac{\Gamma(\frac{1}{2}+\beta z)\Gamma(1+z)\Gamma(-z)}{2\Gamma(1-\beta z)\Gamma(1-\sigma+\alpha z)}
\end{align*}
has a removable singularity at $z=0$ if and only if $\sigma\in\mathbb{N}$.
Thus (\ref{eq:residue vanishes}) holds if $\sigma=1$ and we immediately
obtain
$$
\mathcal{F}\left\{ p_{1,0}(t,\cdot)\right\} =t^{-1}E_{\alpha,0}(-|\xi|^{2\beta}t^{\alpha})=\frac{\partial}{\partial t}E_{\alpha}(-|\xi|^{2\beta}t^{\alpha}).
$$
By Remark \ref{rem:diff Fourier} and Theorem \ref{thm:time derivative},
$$
\frac{\partial}{\partial t}\mathcal{F}\left\{ p(t,\cdot)\right\} =\mathcal{F}\{\frac{\partial p}{\partial t}(t,\cdot)\}=\mathcal{F}\left\{ p_{1,0}(t,\cdot)\right\} .
$$
Thus,
$$
\mathcal{F}\left\{ p(t,\cdot)\right\} (\xi)=E_{\alpha}(-|\xi|^{2\beta}t^{\alpha})+R(\xi).
$$
By (\ref{eq:scaling}),
\begin{align*}
\mathcal{F}\left\{ p(t,\cdot)\right\} (\xi) &=\int_{\mathbb{R}}e^{-\textnormal{i}x\xi}p(t,x)dx =\int_{\mathbb{R}}e^{-\textnormal{i}x\xi}t^{-\frac{\alpha}{2\beta}}p(1,t^{-\frac{\alpha}{2\beta}}x)dx\\
 & =\int_{\mathbb{R}}e^{-\textnormal{i}t^{\alpha/2\beta}x\xi}p(1,x)dx
  =\mathcal{F}\left\{ p(1,\cdot)\right\} (t^{\frac{\alpha}{2\beta}}\xi).
\end{align*}
Then by the Riemann-Lebesgue lemma, $\mathcal{F}\left\{ p(t,\cdot)\right\} (\xi)$
converges to 0 as $t\rightarrow\infty$. This implies $R(\xi)\equiv0$
and we obtain the desired result.

\subsection*{Case 4: $d=1,$ $\gamma=\beta$.}

Now we additionally assume $\sigma+\alpha\in\mathbb{N}$ and take
$\ell_{0}$ so that
$$
\max(-2,-1-\frac{1}{2\beta})<\ell_{0}<0.
$$
Note that every argument in Case 2 holds except that
\begin{align*}
\frac{\mathcal{H}_{\sigma,\beta}(z)}{2\beta(1+z)} & =-\frac{\Gamma(\frac{1}{2}+\beta+\beta z)\Gamma(2+z)\Gamma(-z)}{2\Gamma(1-\beta-\beta z)\Gamma(1-\sigma+\alpha z)(z+1)}
\end{align*}
has a removable singularity at $z=-1$ if $\sigma+\alpha\in\mathbb{N}$.
Thus (\ref{eq:residue vanishes}) holds and we obtain (\ref{eq:goal}).
The theorems are proved.

\end{document}